\documentclass[12pt]{article}
\usepackage{amsmath,amssymb,amsthm,amsfonts}
\usepackage{graphicx,color}

\usepackage{hyperref}

\newtheorem{prop}{Proposition}[section]
\newtheorem{thm}[prop]{Theorem}
\newtheorem{lemma}[prop]{Lemma}
\newtheorem{cor}[prop]{Corollary}
\newtheorem{Defi}[prop]{Definition}
\newtheorem{Rem}[prop]{Remark}
\newtheorem{Exam}[prop]{Example}
\newtheorem{Exer}[prop]{Exercise}
\newtheorem{Expe}[prop]{Experiment}
\newtheorem{Cons}[prop]{Construction}

\newtheorem{Alg}[prop]{Algorithm}
\newtheorem{Prob}[prop]{Research Problem}
\newenvironment{defi}{\begin{Defi} \rm}{\end{Defi}}
\newenvironment{rem}{\begin{Rem} \rm}{\end{Rem}}

\newcommand{\bc }{{\bf c}}

\newcommand{\be }{{\bf e}}

\newcommand{\br }{{\bf r}}
\newcommand{\bs }{{\bf s}}

\newcommand{\bx }{{\bf x}}
\newcommand{\by }{{\bf y}}

\newcommand{\ff}{{\mathbb F}}
\newcommand{\ffa}{{\bar{\mathbb F}}}
\newcommand{\complex}{{\mathbb C}}
\newcommand{\proj}{{\mathbb P}}

\newcommand{\fq}{{\mathbb F}_q}
\newcommand{\fqtwo}{{\mathbb F}_{q^2}}

\newcommand{\fqm}{{\mathbb F}_{q^m}}

\newcommand\cc{\mathcal{C}}
\newcommand\cl{\mathcal{L}}
\newcommand\cm{\mathcal{M}}

\newcommand\co{\mathcal{O}}
\newcommand\cp{\mathcal{P}}
\newcommand\cx{\mathcal{X}}

\newcommand\C{{\cal C}}

\newcommand\E{{\cal E}}
\newcommand\M{{\cal M}}
\newcommand\N{{\cal N}}
\newcommand\X{{\cal X}}
\newcommand\Y{{\cal Y}}

\newcommand\bg{{\mathbf \Gamma}}

\newcommand\tr{\mathrm{tr}}

\title{The extended coset leader weight enumerator of a twisted cubic code}
\author{Aart Blokhuis\thanks{Aart Blokhuis, Department of Mathematics and Computer Science, Eindhoven University of Technology, The Netherlands, a.blokhuis@TUE.nl},
Ruud Pellikaan\thanks{Ruud Pellikaan, Department of Mathematics and Computer Science, Eindhoven University of Technology, The Netherlands, g.r.pellikaan@tue.nl},
Tam\'as Sz\H{o}nyi\thanks{Tam\'as Sz\H{o}nyi, Department Computer Science, E\"{o}tv\"{o}s Lor\'and University, Hungary, tamas.szonyi@ttk.elte.hu}}
\date{}
\begin{document}
\maketitle

\begin{abstract}
The extended coset leader weight enumerator of the generalized Reed-Solomon $[q+1,q-3,5]_q$ code is computed.
The computation is considered as a question in finite geometry.
For this we need the classification of the points, lines and planes in the projective three space under projectivities that leave the twisted cubic invariant.
A line in three space determines a rational function of degree at most three and vice versa.
Furthermore the double point scheme of a rational function is studied.
The pencil of a true passant of the twisted cubic, not in an osculation plane gives a curve of genus one as double point scheme.
With the Hasse-Weil bound on $\fq$-rational points we show that there is a $3$-plane containing the passant.
\end{abstract}

\section{Introduction}
In general the computation of the weight enumerator of a code is hard and even harder so for the coset leader weight enumerator.
Generalized Reed-Solomon codes are MDS, so their weight enumerators are known and its formulas depend only on the size of the finite field,
the length and the dimension of the code.
The coset leader weight enumerator of an MDS code depends on the geometry of the associated projective system of the dual code.
The coset leader weight enumerator of the $\fq$-ary generalized Reed-Solomon codes of length $q+1$ of codimension four is considered,
so its associated projective systems are normal rational curves.
Moreover the coset leader weight enumerators of the extensions of these codes over $\fqm $ are determined.
In case of the $[q+1,q-2,4]_q$ code where the associated projective system consists of the $q+1$ points of an irreducible plane conic,
the answer \cite{jurrius:2015} depends on whether the characteristic is odd or even.
If the associated projective system of the $[q+1,q-3,5]_q$ code consists of the $q+1$ points of a twisted cubic,
the answer is the main result of this paper and depends on $q$ modulo $6$.\\
This result heavily depends on the classifications of the points, lines and planes in $\proj ^3$ under projectivities that leave the twisted cubic invariant.
The classification of points and planes was done in \cite{bruen:1977, hirschfeld:1985} but that classification of the lines is not fine enough for our purpose.
The point-plane incidence matrix is computed and applied to multiple covering codes by \cite{bartoli:2020},
but we need to know whether a given line is contained in a $3$-plane. This is is the main result of this paper.
The refined classification and the plane-line incidence was computed in recent papers by \cite{davydov:2021a,davydov:2021b}.
They computed, except for the class $\co_6$ (true passants not in an osculation plane) the number of times
a line of a given class is contained in a plane of a given class.\\
In our approach we use the relation between rational functions and codimension two subspaces is.
Furthermore the double point scheme $\E_{\varphi} $ of a rational function $\varphi$ is studied in general.
If the rational function $\varphi$ is a separable simple morphism of degree $d$, then $\E_{\varphi} $ is an absolutely irreducible  curve of genus $(d-1)^2$.
In particular the pencil of planes containing a given line, that is a true passant, not in an osculating plane defines a rational function and its double point scheme is a curve of genus $1$.
With the Hasse-Weil bound it is shown that there is a $3$-plane containing a given true passant in case $q\geq 23$.

\section{The coset leader weight enumerator}

The extended coset leader weight enumerator of a code is considered and it is just our aim to determine this enumerator of the code associated to the twisted cubic.
Let $C$ be an $\fq$-linear code of length $n$. Let $\br \in\fq^n$.
The {\em weight of the coset} $\br+C$ is defined by $\mbox{wt}(\br+C)=\min\{\mbox{wt}(\br+\mathbf{c}):\mathbf{c}\in C\}$.
A \emph{coset leader} is a choice of an element $\br\in\fq^n$ of minimal weight in its coset, that is $\mbox{wt}(\br)=\mbox{wt}(\br+C)$.
Let $\alpha _i$ be the number of cosets of $C$ of weight $i$.
The \emph{coset leader weight enumerator} is the polynomial with coefficients $\alpha _i$.\\
A {\em coset leader} decoder gives as output $\br - \be$, where $\br$ is the received word and $\be $ is a chosen coset leader of the coset of $\br$.
So $\br - \be$ is a nearest codeword to $\br $, but sometimes it is not the only one.
The probability of decoding correctly by the coset leader decoder on a $q$-ary symmetric channel with cross-over probability $p$
is computed by means of the coset leader weight enumerator, see \cite[Prop. 1.4.32]{pellikaan:2017}.\\
Let $C$  be an $\fq$-linear code with {\em parameters} $[n,k,d]$,
that is of {\em length} $n$, {\em dimension} $k$ and {\em minimum distance} $d$.
Then $C\otimes \fqm$ is the $\fqm$-linear code generated by $C$ and
it is called the {\em extension code} of $C$ over $\fqm$.
The weight enumerator of such an extension code has coefficients that are polynomials in $q^m$, see \cite{jurrius:2012,jurrius:2013}.
Similarly the {\em extended coset leader weight enumerator} of $C$ has coefficients $\alpha_i(T)$ that are polynomials in $T$ such that
$\alpha_i(q^m)$ is the number of cosets of $C\otimes \fqm$ that are of weight $i$, see \cite{helleseth:1979,jurrius:2015}.
Now $\alpha_i(q^m)$ is divisible by $q^m-1$ for all $m,i\geq 1$, since the coset weight of $\br +C$ and of $\lambda\br +C\otimes \fqm $  with respect to  $C\otimes \fqm $
have the same size for all nonzero $\lambda \in \fqm$.
So also $\alpha_i(T)$ is divisible by $T-1$ for all $i\geq 1$. Define
\[
a_i(T):=\frac{\alpha_i(T)}{T-1}.
\]
Then $a_i(T)= { n \choose i}(T-1)^{i-1}$ for all $1 \leq i\leq (d-1)/2$ and $\sum _{i=0}^{n-k} \alpha_i(T)=T^{n-k}$.
So  $\sum _{i=1}^{n-k} a_i(T)=\sum_{i=0}^{n-k-1}T^i$.

\subsection{Codes versus projective systems}

Let $\fq$ be the field with $q$ elements, where $q=p^h$ for some prime $p$.
The {\em projective space} of dimension $r$ is denoted by $\proj^r$. Let $\ff $ be a field. An $\ff$-{\em rational point} of $\proj^r$ is an
equivalence class of $\ff^{r+1}\setminus \{ 0 \}$ under the equivalence relation $\bx \equiv \by $ if and only if $\bx = \lambda \by $
for some nonzero $\lambda \in \ff$. The equivalence class of $\bx =(x_0, x_1,\ldots ,x_r)$ is denoted by $(x_0: x_1:\ldots :x_r)$.
Dually a {\em hyperplane} in $\proj ^r$ given by the equation $a_0X_1+a_1X_1+\cdots +a_rX_r=0$ is denoted by $[a_0: a_1:\ldots :a_r]$.
Let $\cx  $ be a subvariety of $\proj ^r$.
Then the set of $\ff$-rational points  of $\cx $ is denoted by $\cx (\ff )$ and by $\cx (q)$ in case $\ff=\fq$.\\

Let $H$ be a {\em parity check} matrix of $C$, that is an $(n-k)\times n$ matrix such that $\bc \in C$ if and only if $H\bc^T =0$.
Hence a codeword of weight $w$ corresponds one-to-one to a linear combination of $w$ columns of a given parity check matrix adding up to zero.
The {\em syndrome} $\bs$ (with respect to H) of a received word $\br \in \fq^n$ is the column vector of length $n-k$ defined by $\bs = H\br^T$.
This gives a one-to-one correspondence between cosets and syndromes.
An element of  minimal weight in its coset corresponds one-to-one to a minimal way to write its syndrome as a linear combination of the columns of a given parity check matrix.\\
From now on we assume that the minimum distance of the code is at least $3$, so $H$ has no zero column and no two columns are dependent.
So its columns can be viewed as homogeneous coordinates of $n$ distinct points in projective space of dimension $n-k-1$.\\
More generally, let $H$ be a $l\times n$ matrix of rank $l$ with elements from $\fq$.
We view the columns of $H$ as a {\em projective system} \cite[\S 8.3.2]{pellikaan:2017},
that is a set $\cp$ of $n$ points in projective space $\proj^r(q)$, with $r=l-1$ that do not lie in a hyperplane,
in particular we assume that the columns are non-zero, and no pair is dependent.
We now want to determine $\alpha_i=\alpha_i(q)$ which is the
number of vectors in $\fq^l$ that are a linear combination of some set
of $i$ columns of $H$, but not less. More generally we want to determine
$\alpha_i(q^m)$ which is the number of vectors in $\fqm ^l$
with the same property over $\fqm$ for $i=0,\dots,l$. We think projectively, so for $i=1,\dots,r+1$ we want to determine
$a_i(q^m)$ the number of points in $\proj^{r}(\fqm)$ that lie in a projective subspace of dimension $i-1$ that intersects $\cp$
in exactly $i$ points, and not for smaller $i$.

\subsection{The problem}
We want to answer the question of the coset leader weight enumerator for the extended Reed-Solomon $[q+1,q-3,5]$ code with $4 \times (q+1)$ parity check matrix $H$
whose columns are the vectors $(1,t,t^2,t^3)$ together with $(0,0,0,1)$.
The projective system of $H$ is the normal rational curve of degree $3$ in $\proj^3(q)$, see Section \ref{RNC} in more detail.\\
So, what we want is an answer to the following questions: First for
$\proj^3(q)$ itself, but also for  $\proj^3(q^m)$:\\
$a_1$: How many points belong to the curve $\C_3(q)$? The answer is $q+1$ and also for $m$, since
in our problem $\C_3(q)$ is restricted to $\proj^3(q)$:
\[
a_1(q^m)=q+1 \ \mbox{ and } \ a_1(T)=q+1.
\]
$a_2$: How many points, not already counted under $a_1$,
are on a line containing two points of $\C_3(q)$? This is also
easy. There are $\frac12q(q+1)$ secants, each one of them contributes $q-1$ (for $m$:
$q^m-1$) points, since two secants don't intersect in a point outside $\C_3(q)$, for
that would imply four coplanar points on $\C_3(q)$. So in all cases:
\[
\textstyle  a_2(q^m)=\frac12q(q+1)(q^m-1)\ \mbox{ and } \  a_2(T)=\frac12q(q+1)(T-1).
\]
$a_3$: Now the interesting part starts, how many points are there on a $3$-plane, a plane containing three
points of $\C_3(q)$, not already counted under $a_1$ or $a_2$?
In $\proj^3(q)$ the answer is easy: the rest, so $\frac12q(q+1)^2$. Indeed a point that does
not lie on the curve or on a secant or on a $3$-plane can be used to extend the
arc, but it is well known that the arc is maximal (for $q>3$).\\
Outside $\proj^3(q)$ we argue as follows: If a point is on more than one $3$-plane,
then it must be on a line of $\proj^3(q)$, so forgetting about these points for the
moment, this means that each of the
$\frac16(q+1)q(q-1)$ different $3$-planes contributes
$q^{2m}+q^m+1-(q^2+q+1)-(q^2+q+1)(q^m-q)$ points that are
certainly in this $3$-plane only.\\
What remains is to investigate how many points there are on a line of
$\proj^3(q)$, that is not a bisecant and that is contained in a $3$-plane.\\
Equivalently, we could determine the number of lines that are not
in a $3$-plane. We already know two kinds, tangents, and imaginary chords,
but there will turn out to be many more.\\
A subspace of $\proj^3(q^m)$ (so a point, line or plane)
will be called {\em rational} if it extends a corresponding subspace in $\proj^3(q)$.\\
We start by giving the formula for $a_3(q^m)$ in terms of the parameter
$\mu_q$: the number of lines of $\proj^3(q)$ that are not real chords
and lie in one or more $3$-planes. So $a_3(q^m)$ is equal to
$$
\textstyle \frac12q(q+1)^2+ \frac16(q+1)q(q-1)\left[q^{2m}+q^m+1-(q^2+q+1)(q^m-q+1)\right]+\mu_q(q^m-q).
$$
Hence $a_3(T)$ is equal to
$$
\textstyle  \frac12q(q+1)^2+\frac16(q+1)q(q-1)\left[T^2+T+1-(q^2+q+1)(T-q+1)\right]+\mu_q(T-q).
$$
The first term counts the points $P$ in $\proj^3(q)$, not on $\C_3(q)$ that are either on a tangent of a rational point of $\C_3(q)$ or on an imaginary chord  of 
Assume that the rational point $P$ is not on $\C_3(q)$ and lies on a tangent of a rational point $Q$ of $\C_3(q)$.
Let $\pi_P$ be the projection of the curve $\C_3$ from $P$ to a plane that does not contain the tangent.
Then we get an irreducible plane cubic with  $q+1$ rational points and with a cusp singularity at $\pi_P(Q)$  which is rational.
A $3$-plane going through $P$ corresponds one-to-one by the projection to a line outside $\pi_P(Q)$ that intersects the plane cubic in three rational points.
For every two rational points $P_1$, $P_2$ on the plane cubic, not equal to $\pi_P(Q)$ there is a connecting line that intersects the plane cubic in another point $P_3$, but this point might be equal to $P_1$ or $P_2$ in case the line is tangent at $P_1$ or $P_2$, respectively.
There is one inflexion point, that is a rational point where its tangent intersect with multiplicity $3$.
So there are ${ q \choose 2} -(q-1)={ q-1 \choose 2}$ lines that intersect the plane cubic in three distinct points, none of them equal to $\pi_P(Q)$.
Hence the number of $3$-planes going through $P$ is ${ q-1 \choose 2}$.
If the rational point $P$ lies on an imaginary chord of the place $Q$ of degree $2$, then again we consider the  projection $\pi_P$ of the curve $\C_3$ from $P$ to a plane that does not contain the chord.
We get an irreducible plane cubic with  $q+2$ rational points, $q+1$ are projections of the rational points of $\C_3$ and one is $\pi_P(Q)$ which is a rational ordinary double point with two conjugated imaginary tangents that correspond to the two conjugated imaginary points of $Q$.
A $3$-plane going through $P$ corresponds one-to-one by the projection to a line outside $\pi_P(Q)$ that intersects the plane cubic in three rational points.
There is one inflexion point, that is a rational point where its tangent intersect with multiplicity $3$.
So there are ${ q +1 \choose 2} -q={ q \choose 2}$ lines that intersect the plane cubic in three distinct points, none of them equal to $Q$.
Hence the number of $3$-planes going through $P$ is ${ q \choose 2}$.
The second term is the number of points outside $\proj^3(q)$, in a rational plane, but not on a rational line.
The third term is the number of points outside $\proj^3(q)$, on a rational line that is not a real chord.\\
The rest of these notes are devoted to the determination of the value of $\mu_q$,
that turns out to depend on the value of $q$ mod $6$ and will be given in \S \ref{s-det-muq}.
In order to do that we will give the relation between rational functions and codimension two subpaces of the projective space in Proposition \ref{p-rat-fct-codim2}.
Furthermore we classify several types of lines in Theorem \ref{t-lines}.\\
Finally $a_1(T)+a_2(T)+a_3(T)+a_4(T)=T^3+T^2+T+1$. Hence $a_4(T)$ can be expressed in the known terms  $a_1(T)$, $a_2(T)$ and $a_3(T)$:
$$
a_4(T)=T^3+T^2+T+1-a_1(T)-a_2(T)-a_3(T)
$$

\section{The normal rational curve}\label{RNC}

The {\em normal rational curve}  of degree $r$ is the curve $\C_r$ in $\proj ^r$ with parametric representation
$\{(x^r:x^{r-1}y:\ldots:xy^{r-1}:y^r)\,|\,(x:y)\in \proj ^1\}$, see \cite[\S 21.1]{hirschfeld:1985}.
This map gives an isomorphism of $\proj ^1$ with $\C_r$ and the point $(x:1)$ (and $(1:0)$) on $\proj ^1$ is identified with
$(x^3:x^2:x:1)$ (and $(1:0:0:0)$) on $\C_r$ and both are denoted by $P(x)$ (and $P(\infty )$) where the context makes it clear what is meant.\\
Combinatorially the most important
property of $\C_r$ is that no $r+1$ points are in a hyperplane.
In the following we will take $l=4$ so we have the curve $\C_3$ in $\proj^3$. This curve is
also called the {\em twisted cubic}.
In this dimension the set $\C_r(q)$ is maximal
with respect to the property that no 4 points are coplanar (for $q>3$).

\subsection{The twisted cubic $\C_3$}

Almost everything in this section can be found in  \cite{bruen:1977} and \cite[Chap. 21]{hirschfeld:1985}.\\
The {\em conjugate} of $x\in \bar{\fq}$ is defined by $\bar{x}=x^q$.\\
A {\em chord} is the line joining two points of $\C_3$. We distinguish
{\em real chords}, joining two different points of $\C_3$, {\em tangents}, where the
two points coincide, and {\em imaginary chords}, where the two points are
conjugate points of the extension of $\C_3$ to $\proj^3(q^2)$.\\
An {\em axis} is the line of intersection of two osculating planes.
A {\em real axis} is the intersection of two different osculating planes,
an {\em imaginary axis} is the intersection of two osculating planes at
conjugate points of $\C_3$ in $\proj^3(q^2)$. If the two osculating planes
coincide we obtain a tangent. If $p=3$, then there is exactly one axis, the intersection of all osculating planes
and it is called the axis of $\bg_3$.\\
The {\em tangent} at the point $P(x)=(x^3:x^2:x:1)$ is the
line $\langle (x^3,x^2,x,1),(3x^2,2x,1,0)\rangle$, and in the point
$P(\infty)=(1:0:0:0)$ we have $\langle (1,0,0,0),(0,0,1,0)\rangle$.\\
A {\em passant} or {\em external line} is a line disjoint from $\C_3(q)$,
it is called {\em true} if it is {\em not} an imaginary chord.\\
A {\em unisecant} is a line intersecting $\C_3(q)$ in $1$ point, it is called
{\em true} if it is {\em not} a tangent.\\
A {\em bisecant} or simply {\em secant} is a line intersecting $\C_3(q)$ in
$2$ points (this is the same as a real chord).\\
An {\em $i$-plane}, $i=0,1,2,3$, is a plane containing $i$ points of $\C_3(q)$.\\
A subspace of $\proj^3(q^m)$ (so a point, line or plane)
will be called {\em rational} if it extends a corresponding subspace in
$\proj^3(q)$.\\
A {\em regulus} in $\proj^3(q)$ is the collection of rational lines that are {\em transversals} of three given {\em skew} lines,
that is the collection of lines that intersect three given lines that are mutually disjoint.
The regulus of three skew lines consists of $q+1$ skew lines.
The {\em complementary regulus} of the regulus of three skew lines $l_1, l_2, l_3$,
is the regulus of any three lines $l'_1, l'_2, l'_3$ in the regulus of $l_1, l_2, l_3$.\\

The group $G_q=PGL(2,q)$ of nonsingular $2\times 2$ matrices
$\begin{pmatrix} a & b \\ c & d\end{pmatrix}$ with $ad-bc\not=0$, modulo nonzero multiples of the
identity. So $G_q$ has order $(q^2-1)(q^2-q)/(q-1)=q(q^2-1)$.\\
$G_q$ acts via $\varphi (x:y) =(ax+by : cx+dy)$ on $\proj^1$, also denoted by $\varphi (x)=(ax+b)/(cx+d)$
and it acts sharply $3$-transitively on $\fq\cup\{\infty\}$.
If $\varphi \in G_{q^2}$, that is with coefficients in $\fqtwo$ we
define  the {\em conjugate} of $\varphi $ by $\bar{\varphi}(x)=(\bar{a}x+\bar{b})/(\bar{c}x+\bar{d})$.
Furthermore $G_q$ acts on $\C_3(q)$ and this gives the following map on column vectors:
\[
(x^3,x^2,x,1) \mapsto ((ax+b)^3,(ax+b)^2(cx+d),(ax+b)(cx+d)^2,(cx+d)^3).
\]
This mapping has matrix
\[
\begin{pmatrix}
a^3  & 3a^2b     & 3ab^2     & b^3 \\
a^2c & a^2d+2abc & b^2c+2abd & b^2d \\
ac^2 & bc^2+2acd &ad^2+2bcd  & bd^2 \\
c^3  & 3c^2d     &3cd^2      & d^3
\end{pmatrix},
\]
hence its action extends to a {\em linear collineation} of $\proj^3(q)$.
For $q\ge 5$, $G_q$ is the full group of {\em projectivities} in $\proj^3(q)$ fixing $\C_3(q)$ by  \cite[Lemma 21.1.3]{hirschfeld:1985}.
In \cite[p. 233]{hirschfeld:1985} the action is  on row vectors on the left, whereas
in this paper the action is on column vectors on the right,
since we consider the projective system of the code with the column vectors of the parity check matrix as its points.

\subsection{The classification of planes and points in $\proj^3$}\label{s-class-planes-points}

\begin{prop}\label{p-planes}Under $G_q$ there are five orbits $\N_i$ of planes with $n_i=|\N_i|:$\\[4pt]
$\N_1:$ Osculating planes of $\bg_3(q)$, $n_1=q+1$. \\[3pt]
$\N_2:$ Planes with exactly two points of $\C_3(q)$,  $n_2=q(q+1)$.\\[3pt]
$\N_3:$ Planes with three points of $\C_3(q)$, $n_3=\frac16q(q^2-1)$. \\[3pt]
$\N_4:$ Planes with exactly one point of $\C_3(q)$, not osculating, $n_4=\frac12q(q^2-1)$.\\[3pt]
$\N_5:$ Planes with no points of $\C_3(q)$, $n_5=\frac13q(q^2-1)$.
\end{prop}

\begin{proof}See Corollary 4 of Chapter 21 in \cite{hirschfeld:1985}.
\end{proof}

\begin{rem}\label{r-planes}There is another way to look at this: For the plane $[1:c:b:a]$ consider the
cubic $f(x)=x^3+cx^2+bx+a=(x-\alpha)(x-\beta)(x-\gamma)$. \\
$\N_1$:
If $\alpha=\beta=\gamma$
we have an osculating plane, where $\alpha=\infty$ corresponds to the
plane $[0:0:0:1]$, or $X_3=0$. \\
$\N_2$:
If $\alpha=\beta\ne\gamma$, we have a plane with two points.
The case $\alpha=\beta=\infty$, $\gamma=0$ corresponds to the plane
$[0:0:1:0]$ or $X_2=0$.\\
$\N_3$. If $\alpha,\beta,\gamma$ are different elements from $\fq$ we get a plane with
three points, for $\alpha=\infty$, $\beta=0$, $\gamma=1$
we get $[0:1:-1:0]$, or $X_1=X_2$.\\
$\N_4$.
If $\alpha\in\fq$, $\beta=\bar\gamma\not\in\fq$. If
$\alpha=\infty$ then we have
the plane $[0:1:-t:n]$ for some irreducible polynomial $X^2-tX+n=0$,
with $t=\beta+\bar{\beta}$ and $n=\beta \bar{\beta}$.\\
$\N_5$.
Finally if $f$ is irreducible we have a plane without points of $\C_3(q)$.
\end{rem}

At each point $P(x)=(x^3:x^2:x:1) $ of $\C_3$ we have an {\em osculating plane}
$\pi(x)=[1:-3x:3x^2:-x^3]$ and $\pi(\infty )=[0:0:0:1]$ parameterizing the {\em osculating developable} $\bg_3$.\\
If $q\ne0$ mod $3$,
so if $p\ne3$ then there is an associated {\em null-polarity}
\[
(a_0:a_1:a_2:a_3)\leftrightarrow [-a_3:3a_2:-3a_1:a_0]
\]
interchanging $\C_3$ and $\bg_3$, and their corresponding chords and axes.

\begin{prop}
Under $G_q$ there are five orbits $\M_i$ of points with $m_i=|\M_i|:$\\[4pt]
(i) If $p\ne3$, then \\ 
$\M_1:$ Points on $\C_3(q)$, $m_1=q+1$. \\[3pt]
$\M_2:$ Points off $\C_3(q)$, on a tangent, $m_2=q(q+1)$.\\[3pt]
$\M_3:$ Points on three osculating planes, $m_3=\frac16q(q^2-1)$. \\[3pt]
$\M_4:$ Points  off $\C_3(q)$, on exactly one osculating plane, $m_4=\frac12q(q^2-1)$.\\[3pt]
$\M_5:$ Points on no osculating plane, $m_5=\frac13q(q^2-1)$.\\[3pt]
(ii) If $p=3$, then \\ 
$\M_1:$ Points on $\C_3(q)$, $m_1=q+1$. \\[3pt]
$\M_2:$ Points on all osculating planes, $m_2=q+1$.\\[3pt]
$\M_3:$ Points off $\C_3(q)$, on a tangent, on one osculating plane, $m_3=q^2-1$. \\[3pt]
$\M_4:$ Points  off $\C_3(q)$, on a real chord,  $m_4=\frac12q(q^2-1)$.\\[3pt]
$\M_5:$ Points on an imaginary chord, $m_5=\frac12q(q^2-1)$.
\end{prop}

\begin{proof}See Corollary 5 of Chapter 21 in \cite{hirschfeld:1985}.
\end{proof}

\begin{rem}\label{r-lines}If $p\ne3$, then $\M_2$ is also the set of points on exactly two osculating planes, and
$\M_3 \cup \M_5$ is the set of points not in $\C_3(q)$ on a real (or imaginary) chord,
and   $\M_4$ is the set of points not in $\C_3(q)$ on an imaginary (or real) chord if $q \equiv 1 \mod 3$ (or $q \equiv -1 \mod 3$ respectively)
by the corollary of \cite[Lemma 21.1.11]{hirschfeld:1985}.\\
If $p=3$, then $\M_2 \cup \M_3$ is the set of points not in $\C_3(q)$ on a tangent.\\
Hence for all $q$ we have that every point not in $\C_3(q)$ is on a unique line that is a tangent, a real chord or an imaginary chord.
\end{rem}

\begin{rem}We will give a partition of the lines in $\proj^3$ in Section \ref{s-class-lines}.
\end{rem}

\section{Algebraic curves}

Let $\fq$ be the field with $q$ elements, where $q=p^h$ for some prime $p$.
The {\em projective space} of dimension $r$ is denoted by $\proj^r$. Let $\ff $ be a field. An $\ff$-{\em rational point} of $\proj^r$ is an
equivalence class of $\ff^{r+1}\setminus \{ 0 \}$ under the equivalence relation $\bx \equiv \by $ if and only if $\bx = \lambda \by $
for some nonzero $\lambda \in \ff$. The equivalence class of $\bx =(x_0, x_1,\ldots ,x_r)$ is denoted by $(x_0: x_1:\ldots :x_r)$.
Dually a {\em hyperplane} in $\proj ^r$ given by the equation $a_0x_1+a_1x_1+\cdots +a_rx_r=0$ is denoted by $[a_0: a_1:\ldots :a_r]$.
Let $\X $ be a subvariety of $\proj ^r$.
Then the set of $\ff$-rational points  of $\X$ is denoted by $\X(\ff )$ and by $\X(q)$ in case $\ff=\fq$.\\
For the theory of algebraic curves we will refer to the textbooks \cite{hartshorne:1977,hirschfeld:2008,stichtenoth:1993}.
By an (algebraic) curve we mean an algebraic variety over a field $\ff $ of dimension one, so it is absolutely irreducible.
Most of the time we assume that the curve is nonsingular, unless stated otherwise. The {\em genus} of the curve $\X$ is denoted by $g(\X)$.

\subsection{Divisors on a curve}\label{s-divisors}
Let ${\cal X}$ be a curve over $\fq$.
A {\em place} of a curve $\X$ over the finite field $\fq$ is an orbit under Frobenius of the points of $\X(\fqm )$ of some finite extension $\fqm$ of  $\fq$.
The {\em degree}  of the place $P$ is  the number of points in its orbit and is denoted by $\deg(P)$.
Alternatively a place can be defined as a {\em discrete valuation} of the {\em function field} $\fq(\X)$.\\
The number of points of the projective line that are defined over $\fq $ is equal to $q+1$,
and a place of degree $d$ corresponds one-to-one to a monic irreducible polynomial in $\fq [X]$ of degree $d$.
In particular $\frac12 (q^2-q)$ is the number of places of degree $2$.\\
A {\em divisor} on a curve $\X$ is a formal sum of places $P$ with integer coefficients such that only finitely many are nonzero.
The degree of the divisor $D=\sum_P m_PP$ is defined by $\deg(D) =\sum_P m_P\deg(P)$.
A divisor is called {\em effective} in case all its coefficients are nonnegative.
A divisor $\sum_P m_P$ is called {\em simple } if $m_P=0$ or $m_P=1$ for all  places $P$.

\subsection{Ramified covers}\label{ramify}

For the following we refer to \cite{hartshorne:1977,hirschfeld:2008,stichtenoth:1993}.

\begin{defi}\label{d-ramify}
Consider a  morphism $\varphi : \X \rightarrow \Y$ of the nonsingular absolutely irreducible curves $\X$ and $\Y$ over the field $\ff $.
Then $\ff (\X$), the function field of $\X$ is a finite field extension  of $\ff (\Y$), the function field of $\Y$, via $\varphi$.
The degree of this extension is also called the {\em degree} of $\varphi $ and will be denoted by $\deg (\varphi)$.\\
Let $x$ be a {\em local parameter} at the place $P$ of $\X$. Then $x$ is a generator of $\M_P$, the unique {\em maximal ideal} of $\co_P$,
the {\em local ring} of $\X$ at $P$. Let $y$ be a local parameter at the place $Q=\varphi(P)$ of $\Y$.
Then the local ring of $\Y$ at $Q$  is via $\varphi $ a subring of $\co_P$. In this way we consider $y$ as an element of $\co_P$
and $y= cx^e$ where $c$ is an invertible element of $\co_P$ and $e$ is a non-negative integer that is called the
{\em ramification index} of $\varphi $ at the place  $P$ and is denoted by $e_P(\varphi)$ or  by $e_P$.
The morphism $\varphi$  is said to {\em ramify} at $P$ and $P$ a ramification place of $\varphi $ if  $e_P>1$.
\end{defi}

\begin{prop}\label{p-deg-ram-ind}
If $P$ is a place of $\X$ and $\varphi(P)=Q$, then $Q$ is a place of $\Y$ and $\deg(Q)$ divides $\deg(P)$ and
$\deg(P)/\deg(Q)$ is called the {\em relative degree} and denoted by $\deg(P,Q)$.
If $Q$ is a place of $\Y$, then
$$
\deg (\varphi)  = \sum _{\varphi(P)=Q} e_P\deg(P,Q).
$$
In particular, the fibre $\varphi^{-1}(Q)$ consist of at most $\deg (\varphi)$ places.
\end{prop}

\begin{proof}See \cite[Theorem III.1.11]{stichtenoth:1993}.
\end{proof}

\begin{rem}\label{r-deg-ram-ind}
If $\deg (\varphi)\leq 3$, then $\varphi$ is injective on the set of ramification places by Proposition \ref{p-deg-ram-ind}.
\end{rem}

\begin{defi}
Let $\varphi : \X \rightarrow \Y$ be a {\em separable} morphism between two curves.
The ramification at $P$ is called {\em tame} if the characteristic does not divide $e_P$, otherwise it is called {\em wild}.
The morphism ramifies at finitely many places.
The {\em ramification divisor} of $\varphi $ is defined by
$$
R_{\varphi } = \sum _{P} (e_P-1)P.
$$
\end{defi}

\begin{defi}\label{d-different}
Consider a  morphism $\varphi : \X \rightarrow \Y$.
Let $x$ be a {\em local parameter} at the place $P$ of $\X$. Let $y$ be a local parameter at the place $Q=\varphi(P)$ of $\Y$.
Then $y= cx^e$ where $c$ is an invertible element of $\co_P$ and $e$ is  the ramification index of $\varphi $ at the place  $P$.
Let $y'$ be the derivative of $y$ with respect to the derivation of $x$.
The {\em different exponent} of $\varphi $ at the place $P$ is the smallest $d$ such that $ y' \in \M_ P^d$  and is denoted by $d_P(\varphi )$ or by $d_P$.
The {\em different divisor} of $\varphi $ is defined by
$$
D_{\varphi } = \sum _{P} d_P\deg(P).
$$
\end{defi}

\begin{rem}\label{r-different}
By the Leibniz rule we have
$$
y'=c'x^e +ecx^{e-1}.
$$
Hence $d_P\geq e_P-1$, and $d_P = e_P-1$ if and only if the ramification at $P$ is tame, that is if characteristic of $\ff $ does not divide $e_P$.
If the ramification is wild then $d_P+1$ is not divisible by the characteristic.
\end{rem}

\begin{thm}[Riemann-Hurwitz genus formula]\label{thm-RH} Let $\varphi : \X \rightarrow \Y$ be a separable morphism between curves that is not constant.
Then
$$
2g(\X)-2 = \deg (\varphi) (2g(\Y)-2) + \deg (D_{\varphi }).
$$
\end{thm}

\begin{proof}See \cite[Corollary 2.4]{hartshorne:1977}, \cite[Therorem 7.27]{hirschfeld:2008} and \cite[Theorem III.4.12]{stichtenoth:1993}.
\end{proof}

If the degree  of the morphism $\varphi$ is $1$, then the morphism is an isomorphism and there is no ramification.

\section{Rational functions on the projective line}

\subsection{Equivalence of rational functions}

\begin{defi}\label{d-rational}A {\em rational function} $\varphi: \proj^1  \dashrightarrow \proj^1$  over $\fq $ of degree $d$.
is given by $\varphi (x:y) = (f(x,y): g(x,y))$ where $f(x,y)$ and $g(x,y)$ are homogeneous polynomials of degree $d$.
Let $h(x,y)=\gcd ((f(x,y), g(x,y))$. The divisor defined by $h(,x,y)=0$ is called the {\em base divisor} of $\varphi$ and is denoted by $B_{\varphi}$.
\end{defi}

\begin{defi}Let $\varphi , \psi : \proj^1 \rightarrow \proj^1$ be two rational functions defined over $\ff $.
They are called {\em right} (R) {\em equivalent} if there is an automorphism $\alpha \in PGL(2,\ff)$
such that $\psi =  \varphi  \circ \alpha $,
and {\em left} (L) {\em equivalent} if there is an automorphism $\beta \in PGL(2,\ff)$
such that $\psi =   \beta \circ \varphi $.\\
Furthermore  $\varphi $ and $\psi$ are called  {\em right-left} (RL) {\em equivalent} if there automorphisms $\alpha \in PGL(2,\ff)$  and $\beta \in PGL(2,\ff)$
such that $\psi = \beta \circ \varphi \circ  \alpha$.
If moreover $\beta =\alpha^{-1}$, then $\varphi , \psi $ are called {\em conjugate}.
\end{defi}

\begin{rem}Let $\varphi: \proj^1  \dashrightarrow \proj^1$  be given by $\varphi (x:y) = (f(x,y): g(x,y))$.\\
(1) Let $h(x,y)=\gcd ((f(x,y), g(x,y))$. Then $\varphi$ is a well-defined map outside the zero zet of $h(,x,y)$.
Let $\tilde{f}(x,y)=f(x,y)/h(x,y)$, $\tilde{g}(x,y)=g(x,y)/h(x,y)$ and $\tilde{\varphi} (x:y) = (\tilde{f}(x,y): \tilde{g}(x,y))$.
Then $\tilde{\varphi} $ is a well-defined function on $\proj^1$,
and $\varphi$ and $\tilde{\varphi} $ define the same function outside the zero zet of $h(,x,y)$.
We call $\tilde{\varphi}$ is the {\em associated morphism} of $\varphi $.\\
We will a distinction between the notions of a {\em rational function} $ \proj^1  \dashrightarrow \proj^1$ and a {\em morphism} $\proj^1  \rightarrow \proj^1$.\\
(2) Let $f(x,y)= \sum _{j=0}^d f_jx^{d-j}y^j$ and $g(x,y)= \sum _{j=0}^d g_jx^{d-j}y^j$.
The $ 2\times (d+1)$ matrix  with first row $(f_0,f_1,\dots,f_d)$ and second row $(g_0,g_1,\dots,g_d)$  has rank $s\leq 2$,
then $s\leq d$ and the image of $\varphi $ is contained in a subspace of $\proj^1$ of dimension $s-1$,
that is either $\proj^1$ or a point when $\varphi $ is constant.
Therefore we assume from now on that the image of $\varphi $ is not constant. Hence $d \geq 2$.\\
(3) Under the L-equivalence of the action of $ PGL(2,\ff)$, the projectivities  of $\proj ^1$, we may assume that the
$2\times (d+1)$ matrix is in row reduced echelon form.
(4) The corresponding rational function on the affine line is also denoted by $\varphi$ and is given by $\varphi (x) = f(x) /g(x)$,
where $f(x)$ and $g(x)$ are univariate polynomials  $d= \max \{ \deg(f(x)), \deg(g(x)) \}$.
By (3) we may assume that $d= \deg(f(x))>  \deg(g(x))$, and $f(x)$ and $g(x)$ are monic, and $f_{0e}=0$ where
$e= \deg(g(x))$.
\end{rem}

\begin{rem}\label{r-rational-2}Let $\varphi: \proj^1 \rightarrow \proj^1$ be a separable morphism.
Then $\deg (D_{\varphi }) =2d-2 $ by the Riemann-Hurwitz genus formula \ref{thm-RH}.\\
(1) If $d= \deg(f(x)) > \deg(g(x))$, then $\varphi(P_{\infty})=P_{\infty}$ and the ramification exponent of $P_{\infty}=(1:0)$ is equal to $d-\deg(g(x))$.\\
(2) Let $P=(x_0:1)$ with $x_0 $ in some extension of $\fq$ and $\varphi (x_0)=0$.
Then $\varphi (x) = (x-x_0)^{e_P} \psi(x)$ for some rational function $\psi(x)$ such that $\psi(x_0)\not=0$.
So $\varphi' (x) = e_P(x-x_0)^{e_P-1} \psi(x)+(x-x_0)^{e_P} \psi'(x)$.
Hence $\varphi$ ramifies at $P$, that is $e_P>1$ if and only if $\varphi' (x_0)=0$.
\end{rem}

\subsection{Rational functions versus codimension two subspaces}

\begin{prop}\label{p-rat-fct-codim2}Let $\ff $ be a field with algebraic closure $\ffa$.
Then there is a one-to-one correspondence between $L$-equivalence classes of non-constant rational functions on $\proj ^1$ over  $\ff$ of degree $d$ and
codimension $2$ subspaces of $\proj^d(\ff)$.
Furthermore, the rational function is a morphism if and only if the codimension subspace does not intersect $\cc_d(\ffa)$.
\end{prop}

\begin{proof}The proof for morphisms and $\ff =\complex $ is given in \cite[p. 106]{eremenko:2002} and generalizes for arbitrary fields as follows.\\
Let $\varphi (x:y) = (f(x,y): g(x,y))$ be a non-constant rational function on $\proj ^1$ over  $\ff$ of degree $d$ with
$f(x,y)= \sum _{j=0}^d f_jx^{d-j}y^j$ and $g(x,y)= \sum _{j=0}^d g_jx^{d-j}y^j$  with $f_j, g_j \in \ff $ for all $j$.
Let $\cl_{\varphi }$ be the subspace of  $\proj^d(\ff )$ defined by the homogeneous linear equations $\sum_{j=0}^d f_jX_j=0$ and $\sum_{j=0}^d g_jX_j=0$.
The rational map $\varphi$ is not constant. So $f(x,y)$ and $g(x,y)$ are not a constant multiple of each other.
Hence $\cl_{\varphi }$  is a codimension $2$ subspace of $\proj^d(\ff)$.\\
Conversely, if $f(x,y)$ and $g(x,y)$ have a non-constant factor  $g(x,y)$ in common, then $\cl_{\varphi }$ intersects $\cc_d(\ffa)$ at the zero set of $g(x,y)$.\\
Conversely, let $\cl$ be a codimension $2$ subspace of $\proj^d(\ff)$  by the equations $\sum_{j=0}^d f_jX_j=0$ and $\sum_{j=0}^d g_jX_j=0$.
Define $f(x,y)=\sum_{j=0}^d f_jx^jy^{d-j}$ and $g(x,y)=\sum_{j=0}^d g_jx^jy^{d-j}$.
Then $f(x,y)$ and $g(x,y)$are not a constant multiple of each other, since $\cl$ has codimension $2$.
So $\varphi_{\cl }$ defined by $\varphi_{\cl } (x:y) = (f(x,y):g(x,y))$ is a non-constant rational functions on $\proj ^1$ over  $\ff$ of degree $d$.\\
If $\cl$  intersects $\cc_d(\ffa)$ at $P(x_0:y_0)$, then $f(x_0,y_0)=0$ and $g(x_0,y_0)=0$.
Hence $f(x,y)=(x_0y-y_0x)c(x,y)$ and $g(x,y)=(x_0y-y_0x)d(x,y)$ for some homogeneous polynomials $c(x,y)$  and $d(x,y)$ of degree $d-1$.
Therefore $f(x,y)$ and $g(x,y)$ have a factor in common.\\
If $\psi$ is $L$-equivalent with $\varphi $, then there are $a,b,c,d \in \ff $ such that $ad-bc\not=0$ and
$\psi (x,y) =(a f(x,y)+b g(x,y))/(cf(x,y)+d g(x,y))$. Hence $\cl_{\psi }=\cl_{\varphi }$.\\
Conversely, another pair of homogeneous linear equations defining $\cl$ will give $\psi$, a rational function on $\proj ^1$ over  $\ff$ of degree $d$
that is $L$-equivalent with $\varphi$.
\end{proof}

\begin{rem}The number of intersection points of $\cl_{\varphi }$ with $\cc_d(\ffa)$, counted with multiplicities, that is
the degree of the base divisor of $\varphi $  is equal to
$\deg (\varphi)-\deg (\tilde{\varphi})$, where $\tilde{\varphi}$ is the associated morphism of $\varphi $.
\end{rem}

\begin{rem}\label{r-ramexp-intmult}
Let $\varphi (x) = f(x)/ g(x))$ be a non-constant rational function  of degree $d$ with
$f(x)= \sum _{j=0}^d f_jx^{d-j}$ and $g(x)= \sum _{j=0}^d g_jx^{d-j}$ and $f_j, g_j \in \ff $ for all $j$.
Then  $x\in \varphi^{-1} (u) $ if and only if $P(x)$ is in the hypersurface $H_{\varphi, u}$ with equation $\sum_{j=0}^d (f_j-ug_j)x_j=0$.
More precisely the ramification exponent of $e_x(\varphi)$ is equal to the intersection multiplicity of that hypersurface with $\C_r$.\\
In particular places in the support of $R_{\varphi }$ correspond one-to-one to those places where  $H_{\varphi, u}$ is tangent to $\C_r $ for some $u$.
The hypersurfaces $H_{\varphi, u}$ contain $\cl_{\varphi}$ for all $u$ and they form the so called {\em pencil} of hyperplanes of $\cl_{\varphi}$.
\end{rem}

\begin{rem}\label{r-catalan}Every morphism $\varphi: \proj^1 \rightarrow \proj^1$ of degree $d$ has a different divisor $D_{\varphi }$ that is an effective divisor of degree $2d-2$.
Let $C_d=\frac{1}{d}{ 2d-2 \choose d-1 }$ be the $d$-th Catalan number.
If $\ff $ is an algebraically closed field and $D$ a generic effective divisor of degree $2d-2$, then there are $C_d$ morphisms on $\proj ^1$ of degree $d$
with the given $D$ as different divisor \cite{goldberg:1991}. In particular, there are $2$ morphisms on $\proj ^1$ of degree $3$ with the given effective divisor $D$ of degree $4$ as different divisor.
\end{rem}

\subsection{A partition of morphisms on $\proj^1$ of degree $2$}

\begin{prop}\label{p-deg=2}Let $\varphi : \proj^1 \rightarrow \proj^1$ be a morphism of degree $2$ over $\fq $.\\
Then one of the following cases hold:\\
$(1)$ \ \ \ $q$ is odd and $\varphi$ is separable and tame and $D_{\varphi }=R_{\varphi }$ and\\
$(1.a)$ there are two $\fq$-rational points $P_1$ and $P_2$ such that $R_{\varphi }=P_1+P_2$,\\
$(1.b)$ there is a place $Q$  of degree $2$ such that $R_{\varphi }=Q$,\\
$(2)$ \ \ \ $q$ is even and\\
$(2.a)$ $\varphi$ is purely inseparable,\\
$(2.b)$ $\varphi$ is separable and $R_{\varphi }=P$ and $D_{\varphi }=2P$ for a $\fq$-rational point $P$.
\end{prop}

\begin{proof}If the morphism is not separable, then the characteristic divides the degree of $\varphi $.
Hence the characteristic is $2$ and the map is purely inseparable.
If the morphism is separable, then $deg (D_{\varphi })=2$ by Remark \ref{r-rational-2}.
Furthermore the ramification index is $2$ at every place where $\varphi$ ramifies by Proposition \ref{p-deg-ram-ind}.\\
$(1)$ If the characteristic is odd, then the ramification index is $2$ at the ramification places,
which is not divisible by the characteristic. Hence $D_{\varphi }=R_{\varphi }$ and has degree $2$. So either\\
$(1.a)$ there are two $\fq$-rational points $P_1$ and $P_2$ such that $R_{\varphi }=P_1+P_2$, or\\
$(1.b)$ there is a place $Q$ of degree $2$ such that $R_{\varphi }=Q$.  \\
$(2)$ If the characteristic is even, then either\\
$(2.a)$ $\varphi$ is purely inseparable,\\
$(2.b)$ or $\varphi$ is separable and ramifies at a place $P$. Then $e_P=2$ by Proposition \ref{p-deg-ram-ind} and the ramification is wild and
$2=e_P \leq d_P \leq \deg D_{\varphi }=2$. So $P$ is an $\fq$-rational point and $R_{\varphi }=P$ and $D_{\varphi }=2P$
\end{proof}

\begin{rem}\label{r-normal-deg=2}Without proof we mention that all the cases given in Proposition \ref{p-deg=2}
do appear and are $RL$-equivalent to one of the following normal forms:\\
$(1.a)$ $\varphi (x)=x^2$ and $R_{\varphi }=P(0)+P(\infty)$ with $P(0)=\varphi(P(0))$ and $PP(\infty)=\varphi(P(\infty))$.\\
$(1.b)$ $\varphi (x)=(x^2+d)/x$ where $d$ a chosen non-square in $\fq $ and
$R_{\varphi }=Q$ with $Q$ the place of degree $2$ corresponding to the irreducible polynomial $X^2-d$.\\
$(2.a)$ $\varphi (x)=x^2$ where  $\varphi$ is purely inseparable. \\
$(2.b)$ $\varphi (x)=x^2/(x+1)$ where $\varphi$ is separable and $R_{\varphi }=P(0)$ and $D_{\varphi }=2P(0)$.
\end{rem}

\begin{defi}\label{d-pijk}Let $\varphi : \proj^1 \rightarrow \proj^1$ be a morphism.
Denote by $p_{i,j,k}$ the number of places $Q$ of $\proj^1$ of degree $i$
that have $j$ places of degree $k$ in $\varphi^{-1}(Q)$.
\end{defi}

\begin{rem}\label{r-pijk}
If $p_{i,j,k}$ is not zero, then $i$ divides $k$ and $ j \leq \deg (\varphi)$ Proposition \ref{p-deg-ram-ind}.
\end{rem}

\begin{prop}\label{p-deg=2pijk}Let $\varphi : \proj^1 \rightarrow \proj^1$ be a separable morphism of degree $2$ over $\fq $.
Then corresponding to those given in Proposition \ref{p-deg=2} the following holds:\\
$(1.a)$ $p_{1,1,1}=2$, $p_{1,2,1}=p_{1,1,2}=\frac12(q-1)$.\\
$(1.b)$ $p_{1,1,1}=0$, $p_{1,2,1}=p_{1,1,2}=\frac12(q+1)$.\\
$(2.a)$ $p_{1,1,1}=q+1$, $p_{2,1,2}=\frac12(q^2-q)$.\\
$(2.b)$ $p_{1,1,1}=1$, $p_{1,2,1}=p_{1,1,2}=\frac12q$.
\end{prop}

\begin{proof}The cases correspond to those given in Proposition \ref{p-deg=2}.\\
$(1.a)$. We have that $R_{\varphi }=P_1+P_2$. So $p_{1,1,1}=2$.
Every rational point $P$ of $\proj^1$ is mapped to a rational point of $\proj^1$, and  $\varphi^{-1}(Q)$
has at most $2$ rational points for every rational point $Q$ of $\Y$, since $\deg (\varphi )=2$. So $p_{1,1,1}+2p_{1,2,1}=q+1$.
For every rational point $Q$ of $\Y$  we have that $\varphi^{-1}(Q)$
consists either of one ramification point or two rational points or one place of degree $2$.
Hence  $2+p_{1,2,1}+p_{1,1,2}=q+1$. Combining these two equations gives the result.\\
The other cases are treated similarly.
\end{proof}

\section{The double point scheme of a morphism}

Let $\varphi: \proj^1  \dashrightarrow \proj^1$ be a non-constant rational function of degree $d$ with $\varphi (x) = f(x)/g(x)$.
Suppose there exist $x, y \in \ff $ such that $x\not=y$ and $\varphi(x)=\varphi(y)$.
Then $\frac{ \varphi(x) - \varphi(y)}{x-y}=0$. So $(f(x)g(y)-f(y)g(x))/(x-y) =0$.

\begin{defi}Let $\varphi : \proj^1 \dashrightarrow \proj^1$ be a rational function of degree $d$ with $\varphi (x) = f(x)/g(x)$.
The {\em double point polynomial } $\Delta_{\varphi } $ of  $ \varphi $ is defined by
$$
\Delta_{\varphi } (x,y) = \frac{ f(x)g(y)-f(y)g(x)}{x-y}.
$$
\end{defi}

\begin{rem}\label{r-permutation-fct}A {\em permutation rational functions} is a rational morphism $\varphi : \proj^1 \rightarrow \proj^1$ defined over $\fq $
such that the map on the $\fq$-rational points is a permutation. Clearly $\varphi $ is a permutation rational function if and only if
$\E_{\varphi }$ has no points $\fq$-rational points outside the diagonal.
Similarly, a polynomial $f(x) \in \fq [x]$ is called {\em permutation polynomial} if $f$ induces a permutation on $\fq$.
In \cite{hou:2020} the polynomial $F(x,y)=[f(x+y)g(x)-f(x)g(x+y)]/y$ is defined for a rational function $\varphi (x)=f(x)/g(y)$.
Now $\Delta_{\varphi }(x,y)=F(x,y-x)$.
\end{rem}

\begin{rem}Let $h(x)=\gcd(f(x),g(x))$ and $\tilde{f}(x)=f(x)/h(x)$ and $\tilde{g}(x)=g(x)/h(x)$.
Then $\tilde{\varphi}(x)=\tilde{f}(x)/\tilde{g}(x)$ is a morphism, that is $\tilde{f}(x)$ and $\tilde{g}(x))$ are relatively prime.
Furthermore $\Delta_{\varphi } (x,y) =h(x)h(y)\Delta \tilde{\varphi}(x,y)$.
\end{rem}

\begin{rem}The double point polynomial of $\varphi $ is a symmetric bivariate polynomial of bidegree at most $(d-1,d-1)$.
The bihomogenization of the double point polynomial of $\varphi $ is defined by
$$
\Delta_{\varphi } (x_0,x_1,y_0,y_1) = \sum_{0\leq i,j \leq d-1} a_{ij}x_0^{d-1-i}x_1^i y_0^{d-1-j}y_1^j,
$$
where $\Delta  \varphi (x,y) = \sum_{0\leq i,j \leq d-1} a_{ij}x^iy^j$.\\
Then $\Delta_{\varphi } (x_0,x_1,y_0,y_1)$ is a symmetric bivariate, bihomogeneous polynomial of bidegree $(d-1,d-1)$.
\end{rem}

\begin{defi}Let $\E_{\varphi } $ be the subscheme of $\proj^1 \times \proj ^1$ defined by the ideal generated by $\varphi (x_0,x_1,y_0,y_1)$.
It is called the {\em double point scheme } of $\varphi $.
See \cite[Definition V-41]{eisenbud:2000}.
\end{defi}

\begin{lemma}\label{l-quotient-rule}Let $\varphi (x) =f(x)/g(x)$ be a  rational function. Then\\
$(1)$ $\Delta_{\varphi }(x,x) = f'(x)g(x)-f(x)g'(x)$,\\
$(2)$ $(x,x)\in \E_{\varphi }(\ffa)$ if and only if $\varphi $ ramifies at $x$.
\end{lemma}

\begin{proof}$(1)$ Proved similarly as in Calculus.\\
$(2)$ $\Delta_{\varphi }(x,x) = f'(x)g(x)-f(x)g'(x)$ is the numerator of the derivative $\varphi'(x)$.\\
Therefor $(x,x)\in \E_{\varphi }(\ffa)$ if and only if $\Delta_{\varphi }(x,x) = 0$ if and only if $\varphi'(x)=0$ if and only if $\varphi $ ramifies at $x$.
\end{proof}

\begin{defi}
The ramification at $P$ is called {\em simple} if $e_P=2$.
The morphism $\varphi $ is called simple if all its ramification points are simple and if $\varphi $ ramifies at distinct places $P_1$ and $P_2$,
then $\varphi (P_1)$ and $\varphi (P_2)$ are distinct.
\end{defi}

\begin{defi}Let $\pi_1 : \E_{\varphi } \rightarrow \proj^1$ be the projection on the first factor and
$\pi_2 : \E_{\varphi } \rightarrow \proj^1$ the projection on the second factor.
\end{defi}

\begin{prop}\label{p-simple}Let $\varphi : \proj^1 \rightarrow \proj^1$ be a separable simple morphism.
Then $\E_{\varphi } $ is reduced and nonsingular.
\end{prop}

\begin{proof}
$\E_{\varphi } $ is contained in $\proj^1 \times \proj^1$ and is defined by one equation. So it has no embedded components.
Hence if it nonsingular, then it is reduced. Therefor it is sufficient to show that $\E_{\varphi } $ is nonsingular. Let $P=(a,b) \in \E_{\varphi }(\ffa)$. \\
Let $\cl =\pi_1^{-1}(a)$ and $\cm =\pi_2^{-1}(b)$.
If one of the  intersection multiplicities $I(P;\cl, \E_{\varphi })$ or $I(P;\cm, \E_{\varphi })$ is $1$, then $\E_{\varphi }$ is nonsingular at $P$.
Furthermore  $e_P(\pi_1)=I(P;\cl, \E_{\varphi })$ and $e_P(\pi_2)=I(P;\cm, \E_{\varphi })$ holds for the ramification exponents as in \ref{r-ramexp-intmult}.\\
$(1)$ If $a=b$, then $\varphi $ ramifies at $a$ by Proposition \ref{l-quotient-rule} with exponent $2$, since $\varphi $ is simple.
So $\varphi(x)=(x-a)^2f(x)/g(x)$ and $f(0)\ne 0 \ne g(0)$.
Hence $\Delta_{\varphi } (a,y)= -(y-a)^2f(y)g(0)/(a-y)=(y-a)f(y)g(a)$ and its multiplicity at $y=a$ is $1$, since $f(a)\ne 0 \ne g(a)$.
Therefor $I((a,a);\cl, \E_{\varphi })=e_{(a,a)}(\pi_1)=1$ and $\E_{\varphi }$ is nonsingular at $(a,a)$.\\
$(2)$ If $a\ne b$ and $\varphi $ does not ramify at $a$ and also not at $b$, then $\varphi(x)=(x-a)(x-b)f(x)/g(x)$ and $f(a)\ne 0 \ne g(a)$ and $f(b)\ne 0 \ne g(b)$.
Hence $\Delta_{\varphi } (a,y)= -(y-a)(y-b)f(y)g(a)/(a-y)=(y-b)f(y)g(a)$ and its multiplicity at $y=b$ is $1$, since $f(b)\ne 0 \ne g(a)$.
Therefor  $I((a,b);\cl, \E_{\varphi })=e_P(\pi_1)=1$ and $\E_{\varphi }$ is nonsingular at $(a,b)$.\\
$(3)$ If $a\ne b$ and $\varphi $ ramifies at $a$ or $b$, then not at both, since $\varphi $ is simple.
We may assume by symmetry of $\Delta_{\varphi }(x,y)$ in $x$ and $y$ that $\varphi $ ramifies at $a$ and not at  $b$.
So $\varphi(x)=(x-a)^2(x-b)f(x)/g(x)$ and $f(a)\ne 0 \ne g(a)$ and $f(b)\ne 0 \ne g(b)$.
Hence $\Delta_{\varphi } (a,y)= -(y-a)^2(y-b)f(y)g(a)/(a-y)=(y-a)(y-b)f(y)g(a)$ and its multiplicity at $y=b$ is $1$, since $a \ne b$ and $f(b)\ne 0 \ne g(a)$.
Therefor  $I((a,b);\cl, \E_{\varphi })=e_P(\pi_1)=1$ and $\E_{\varphi }$ is nonsingular at $(a,b)$.\\
Therefor $\E_{\varphi }$ is nonsingular.
\end{proof}

\begin{prop}\label{p-genus}Let $\varphi : \proj^1 \rightarrow \proj^1$ be a separable simple morphism of degree $d\geq 2$.
Then $\E_{\varphi } $ is an  absolutely irreducible nonsingular curve of genus $(d-2)^2$.
\end{prop}

\begin{proof}
Let $\varphi : \proj^1 \rightarrow \proj^1$ be a separable map of degree $d\geq 2$ with simple ramification.
Then $\E_{\varphi } $ is reduced and nonsingular by Proposition \ref{p-simple} and of bidegree $(d-1,d-1)$.\\
Suppose $\E_{\varphi } $  is reducible over the algebraic closure.
Then it is the union of $\X$ and $\Y$, say  of bidegrees $(a_1,a_2)$ and $(d-1-a_1,d-1-a_2)$, respectively such that
$(a_1,a_2)\not=(0,0)$ and $(a_1,a_2)\not=(d-1,d-1)$. Without loss of generality we may assume that $\X$ and $\Y$ have no component in common.
So $\deg(\X \cdot \Y)=a_1(d-1-a_2)+a_2(d-1-a_1)>0$  according to the Theorem of B\'{e}zout for the product of projective spaces
\cite[Chapter IV, \S 2.1]{shafarevich:1974} as mentioned in Section \ref{s-divisors}.
Hence $\X $ and $\Y$ have a point in common over the algebraic closure.
So $\E_{\varphi } $ is singular at that point, which is a contradiction.
Therefore $\E_{\varphi } $ absolutely irreducible, that is irreducible over the algebraic closure.\\
A non-singular curve in $\proj^1\times \proj^1$ of bidegree $(m,n)$ has genus $(m-1)(n-1)$.
This is shown by the adjunction formula for a curve on a surface, see \cite[Chapter V, Example 1.5.2]{hartshorne:1977}.
Hence $\E_{\varphi } $ has genus $(d-2)^2$.
\end{proof}

\begin{cor}\label{c-elliptic}Let $\varphi : \proj^1 \rightarrow \proj^1$ be a separable simple morphism of degree $3$.
Then $\E_{\varphi } $ is an  absolutely irreducible nonsingular curve of genus $1$.
\end{cor}

\begin{proof}
This is a special case of Proposition \ref{p-genus}. See also \cite{bos:1987}.
\end{proof}

\begin{rem}\label{r-elliptic}
Let $\varphi : \proj^1 \rightarrow \proj^1$ be a separable simple morphism of degree $3$.
Then $\E_{\varphi } $ is an  absolutely irreducible nonsingular curve of genus $1$.
Any $\fq$-rational point of $\E_{\varphi } $ outside the diagonal gives a pair $(x,y)$
such that $x\ne y$ and $\varphi(x) =\varphi (y)$ and $(x,x)$ and $(y,y)$ not in  $\E_{\varphi } $.
Hence there is a third point $z$, distinct from $x$ and $y$  such that $(x,z)\in \E_{\varphi } $.
$\varphi(x) =\varphi (y)=\varphi (z)$.
The number of $\fq$-rational points of $\E_{\varphi } $ on the diagonal is at most $\deg (D_{\varphi })=4$ by by Lemma \ref{l-quotient-rule}.
For every $(x,x)\in \E_{\varphi } $ there exists a $y$ such that $(x,y), (y,x)\in \E_{\varphi } $.
So we have to exclude for every at most $12$ points from $\E_{\varphi } (\fq )$.
The Hasse-Weil bound \cite[\S 5.2]{stichtenoth:1993} gives $|\E_{\varphi }|\geq q+1-2\sqrt{q} $.
Therefor, if $q\geq 23$, then $|\E_{\varphi }|>12$ and there is a triple $(x,y,z)$ of mutually distinct $\fq$-rational  points of $\proj^1$
such that $\varphi(x) =\varphi (y)=\varphi (z)$.
\end{rem}

\section{Lines in $\proj^3$}\label{s-class-lines}

We start by repeating the observation of Remark \ref{r-lines}.\\
Two chords do not intersect in a point outside $\C_3(q)$,
as a consequence, every point (not in $\C_3(q)$) is contained in a
unique chord.\\
If $p\ne 3$ then we also have the dual statement: Two axes can only be
coplanar in an osculating plane, every non-osculating plane contains
exactly one axis.\\

Let us determine the chord through $(x^3:x^2:x:1)$ and
$(y^3:y^2:y:1)$. There are three cases: $x=y \in\fq \cup\{\infty\}$ and
we have a tangent, or $x\ne y$ in $\fq$ and we have a real chord,
or $y=\bar x\in\ff_{q^2}\setminus\fq $ and we have an imaginary chord.\\

Let $n=xy$ and $t=x+y$ (`Norm' and `Trace' in the imaginary case). An easy
computation shows that the chord is
\begin{align*}
c(x,y)=&\langle (-nt,-n,0,1),(t^2-n,t,1,0)) \rangle  \text{~~if~~}
x,y\ne\infty \\
c(\infty,y)=&\langle (1,0,0,0),(0,y^2,y,1)\rangle  \text{~~if~~} y\ne\infty \\
c(\infty,\infty)=&\langle (1,0,0,0),(0,1,0,0)  \rangle
\end{align*}
The chord $c(x,y)$ is a tangent, or a real chord, or an imaginary chord if
the polynomial $X^2-tX+n$ is a square, or reducible but not a square, or
irreducible, respectively. In other words, if $t^2-4n$ is $0$, or a square,
or a non-square, respectively  if  $q$ is odd;
and $t=0$, or $\tr_2(n/t^2)=0$, or $\tr_2(n/t^2)=1$, respectively  if  $q$
is even.\\

We can easily check that every point not in $\C_3(q)$ is on a unique chord:\\
$(1:0:0:0)$ belongs to $\C_3(q)$;\\
$(w:1:0:0)$ belongs to $c(\infty,\infty)$;\\
$(w:v:1:0)$ belongs to $c(x,y)$, where $x+y=v$ and $xy=v^2-w$;\\
$(w:v:u:1)$ belongs to $c(x,y)$, where $x+y=(uv-w)/(u^2-v)$ and \\
$xy=(uw-v^2)/(v-u^2)$ if $v\not=u^2$; and belongs to $c(\infty,u)$ if $v=u^2$.\\

Next, we determine the axis that is the intersection of the osculating
planes $[1:-3x:3x^2:-x^3]$ and $[1:-3y:3y^2:-y^3]$. Again there are three
cases:
$x=y$ and we have a tangent, or $x\ne y\in \fq$ and we have a real axis, or
$y=\bar x\in\fqtwo \setminus\fq $ and we have an imaginary axis.
Similarly to the previous computation, it is easy to check that
\begin{align*}
a(x,y)=&\langle (-3nt,n-t^2,0,3),(3n,t,1,0) \rangle  \text{~~if~~}
x,y\ne\infty \\
a(\infty,y)=&\langle (-3y^2,0,1,0),(3y,1,0,0)\rangle  \text{~~if~~} y\ne\infty
\\
a(\infty,\infty)=&\langle (1,0,0,0),(0,1,0,0)  \rangle
\end{align*}

\subsection{A partition of lines in $\proj^3$}

From \cite{hirschfeld:1985} we follow the description of the different kinds
of lines and refine the classification. The terminology is slightly different,
it is explained in the beginning of the proof.

\begin{thm}\label{t-lines}
Let $P_a=(a:1)$ for $a \in \fq$  and $P_{\infty }=(1:0)$.
Let $Q_d$ be the place of degree $2$ given by the irreducible polynomial  $x^2-d$, where $d \in \fq^* $ is a non-square.
Choose a fixed irreducible polynomial $x^2+x+n$, that is with discriminant $1-4n$ being a non-square in case $q$ is odd,
and $tr(n)=1$ if $q$ is even. Let $Q$ be the place of degree $2$ given by the irreducible polynomial  $x^2+x+n$.\\
The set of lines of $\proj^3(q)$ are partitioned in the table below.
The parameters $u$, $v$ and $d$ in the table are fixed and chosen such that $u$, $3v$ and $d$ are non-squares.
Some classes only occur for characteristic $2$ or $3$ and are indicated by $\co_i(2)$ and $\co_i(3)$, respectively,
and $\co_i'$ is not defined in characteristic $3$.
All classes except $\co_6$ form orbits under the action of $G_q$.\\
For every orbit a representative line $\cl $, the corresponding rational function $\varphi=f(x)/g(x)$, the base divisor $B_{\varphi}$  of $\varphi $,
and the ramification divisor $R_{\tilde{\varphi}}$ and different divisor $D_{\tilde{\varphi} }$ of the  associated morphism $\tilde{\varphi}$ are given.
The two vectors generating the line $\cl $ are given in the first and second row of the corresponding class.
The $f(x)$ and $R_{\tilde{\varphi}}$ are given in the first row, and $g(x)$ and $D_{\tilde{\varphi} }$ in the second row.\\
The unisecant, osculating and plane are abbreviated by unisec., oscul. and pl., respectively.
\end{thm}
\newpage
\begin{center}
\[
\!\!\!\!\!\!\!\!\!\!\!\!\!\!\!\!\!\!\!\!\!\!\!\!\!\!\!\!
\begin{array}{|l||l|c|c|c|c|c|}
\hline
\text{Class} &\text{Name}                    &\text{Size}          & \cl     &\varphi(x)=f(x)/g(x) & B_{\varphi}   & R_{\tilde{\varphi}}; D_{\tilde{\varphi}}\\
\hline
\hline
\co_1         &\text{Real chords}            & \frac12q^2+\frac12q & (1,0,0,0) & x^2               &  P_0+P_{\infty}& 0\\
              &                              &                     & (0,0,0,1) & x                 &               & 0\\
\co_1'        &\text{Real axes}              & \frac12q^2+\frac12q & (0,1,0,0) & x^3               &  0            & 2P_0+2P_{\infty}\\
              &                              &                     & (0,0,1,0) & 1                 &               & 2P_0+2P_{\infty}\\
\co_2         &\text{Tangents}               & q+1                 & (0,0,1,0) & x^3               & 2P_0          & 0 \\
              &                              &                     & (0,0,0,1) & x^2               &               & 0 \\
\co_3         &\text{Imaginary}              &\frac12q^2-\frac12q  & (n,-n,0,1) & x^3+(n-1)x-nx     & Q             & 0\\
              &\text{chords}                 &                     & (1-n,-1,1,0) & x^2+x+n           &               & 0\\
\co_3'        &\text{Imaginary}              &\frac12q^2-\frac12q  & (3n,n-1,0,3) & x^3-3nx-n         & 0             & 2Q\\
              &\text{axes}                   &                     & (3n,-1,1,0) & x^2+x+\frac13(1-n)&               & 2Q\\
\co_4         &\text{True unisec.}           & q^2+q               & (0,1,0,0) & x^3               & P_0           & P_0+P_{\infty}\\
              &\text{in oscul. pl.}          &                     & (0,0,0,1) & x                 &               & P_0+P_{\infty}\\
\co_4^-(2)    &\text{True unisec.}           & q+1                 & (0,1,0,0) & x^3               & P_0           &\mbox{purely} \\
              &\text{in oscul. pl.}          &                     & (0,0,0,1) & x                 &               &\mbox{inseparable}\\
\co_4^+(2)    &\text{True unisec.}           & q^2-1               & (0,1,1,0) & x^3               & P_0           & P_0\\
              &\text{in oscul. pl.}          &                     & (0,0,0,1) & x^2+x             &               & 2P_0\\
\co_5^-       &\text{Unisec. not}            &\frac12q^3-\frac12q  & (-d,0,1,0) & x^3+dx            & P_0           & Q_d\\
              &\text{in oscul. pl.}          &                     & (0,0,0,1) & x^2               &               & Q_d\\
\co_5^+       &\text{Unisec. not}            &\frac12q^3-\frac12q  & (-1,0,1,0) & x^3+x             & P_0           & P_1+P_{-1}\\
              &\text{in oscul. pl.}          &                     & (0,0,0,1) & x^2               &               & P_1+P_{-1}\\
\co_5(2)      &\text{Unisec. not}            & q^3-q               & (1,0,1,0) & x^3+x             & P_0           & P_1\\
              &\text{in oscul. pl.}          &                     & (0,0,0,1) & x^2               &               & 2P_1\\
\co_5'^-      &\text{Passants in}            &\frac12q^3-\frac12q  & (0,0,1,0) & x^3               &  0            & Q_{3v}+2P_0\\
              &\text{oscul. pl.}             &                     & (0,v,0,1) & x^2-v             &               & Q_{3v}+2P_0\\
\co_5'^+      &\text{Passants in}            &\frac12q^3-\frac12q  & (0,0,1,0) & x^3               &  0            & P_1+P_{-1}+2P_0\\
              &\text{oscul. pl.}             &                     & (0,\frac13,0,1) & x^2-\frac13       &               & P_1+P_{-1}+2P_0\\
\co_5'(2)     &\text{Passants in}            & q^3-q               & (0,0,1,0) & x^3               &  0            & P_1+2P_0\\
              &\text{oscul. pl.}             &                     & (0,1,0,1) & x^2+1             &               & 2P_1+2P_0\\
\co_6         &\text{Passants not}           & q^4-q^3            &           &                   &  0            &\mbox{simple}\\
              &\text{in oscul. pl.}          & -q^2+q              &           &                   &               & \\
\co_7(3)      &\text{Axis of $\bg_3$}        & 1                   & (0,1,0,0) & x^3               & 0             &\mbox{purely} \\
              &                              &                     & (0,0,1,0) & 1                 &               &\mbox{inseparable}\\
\co_{8.1}^-(3)&\text{Passants}               &\frac12q^2-\frac12   & (0,1,0,0) & x^3-ux            & 0             & 2P_{\infty}\\
              &\text{meeting axis}           &                     & (u,0,1,0) & 1                 &               & 4P_{\infty}\\
\co_{8.1}^+(3)&\text{Passants}               &\frac12q^2-\frac12   & (0,1,0,0) & x^3-x             & 0             & 2P_{\infty}\\
              &\text{meeting axis}           &                     & (1,0,1,0) & 1                 &               & 4P_{\infty}\\
\co_{8.2}(3)  &\text{Passants}               & q^3-q               & (1,1,0,0) & x^3-x^2           & 0             & P_0+2P_{\infty}\\
              &\text{meeting axis}           &                     & (0,0,1,0) & 1                 &               & P_0+3P_{\infty}\\
\hline
\end{array}
\]
\end{center}

\begin{proof}Everything is shown in Lemma 21.1.4 of \cite{hirschfeld:1985},
except the subdivisions of $\co_4$, $\co_5$, $\co_5'$ and $\co_8$, and the statements about the rational functions.
We use the term true unisecant for non-tangent lines that intersect $\C_3(q)$ in exactly one point.
Similarly, for external lines we also use the term passant, and such a line is called a true passant if it is not a chord.\\
Every representative line $\cl $ of an orbit is given by two vectors, that is by a $2 \times 4 $ matrix $L$ of rank $2$.
Let $H$ be the $2 \times 4 $ matrix in row reduced echelon such that form $LH^T=0$.
Then the rows of $H$ give the coefficients of equations of the line $\cl $, and  the rational function $\varphi _{\cl }$ by Proposition \ref{p-rat-fct-codim2}.
\indent
$\co_1$: real chords form a single orbit. A representative of a line in this orbit is given by
$\cl=c(0,\infty)=\langle (1,0,0,0),(0,0,0,1) \rangle $. So $H$ has rows $(0,1,0,0),(0,0,1,0)$.
Hence  $\varphi(x)=x^2/x$, $\tilde{\varphi}(x)=x$,  and $\varphi $ has base divisor $P(0)+P_{\infty}$,  and
$R_{\tilde{\varphi}}= D_{\tilde{\varphi}}=0 $.\\
\indent
$\co_1'$: real axes form a single orbit ($p\not=3$). So it suffices to consider a
particular line $\cl=a(0,\infty)=\langle(0,1,0,0),(0,0,1,0)\rangle $. So $H$ has rows $(1,0,0,0),(0,0,0,1)$.
Hence  $\varphi(x)=\tilde{\varphi}(x)=x^3$, and $R_{\varphi }= D_{\varphi }=2P(0)+2P_{\infty}$.\\
\indent
$\co_3$: imaginary chords form a single orbit with representative $\cl=c(\xi,\bar{\xi})=\langle (n,-n,0,1),(1-n,-1,1,0) \rangle $,
where $\xi,\bar{\xi}$ are the zeros of $X^2+X+n$. So $H$ has rows $(1,0,n-1,-n),(0,1,1,n)$.
Hence $\varphi(x)=(x^3+(n-1)x-nx)/(x^2+x+n)$ and $\tilde{\varphi}(x)=x-1$, and $\varphi $ has base divisor $Q$, and
$R_{\tilde{\varphi}}= D_{\tilde{\varphi}}=0 $. \\
\indent
$\co_3'$: imaginary axes form a  single orbit ($p\not=3$) with representative  $\cl=a(\xi,\bar{\xi})=\langle (3n,n-1,0,3),(3n,-1,1,0) \rangle $.
So $H$ has rows $(1,0,-3n,-n),(0,1,1,\frac13(1-n))$.
Hence $\varphi(x)=\tilde{\varphi}(x)=(x^3-3nx-n)/(x^2+x+\frac13(1-n))$ and $R_{\varphi }= D_{\varphi }=2Q$.\\
\indent
$\co_2$, $\co_4$ and $\co_5$: unisecants.\\
It is sufficient to look at the unisecants through $P(0)=(0:0:0:1)$ So,
we may apply elements from the stabilizer of $P = P(0)$ in $G_q$, that is $G_{q,P}$.
This subgroup consists of the matrices
$$
M_{a,c}=
\begin{pmatrix}
a^3  & 0    & 0   & 0 \\
a^2c & a^2  & 0   & 0 \\
ac^2 & 2ac  & a   & 0 \\
c^3  & 3c^2 & 3c  & 1
\end{pmatrix}.
$$
There are three cases: $\co_2$ unisecants that are tangent, $\co_4$ unisecants in
an osculating plane, and $\co_5$ unisecants not in an osculating plane. This is
the partition in \cite[Lemma 21.1.4]{hirschfeld:1985}, we are going to refine this.\\
The first case is that the line $\langle (0,0,0,1),(0,0,1,0)\rangle$ is mapped to itself. This gives:\\
\indent
$\co_2$: tangents with representative $\cl=\langle (0,0,0,1),(0,0,1,0)\rangle$.
These lines form a single orbit (of size $q + 1$).
So $H$ has rows $(1,0,0,0),(0,1,0,0)$.
Hence  $\varphi(x)=x^3/x^2$, $\tilde{\varphi}(x)=x$, and $\varphi $ has base divisor $2P(0)$, and
$R_{\tilde{\varphi}}= D_{\tilde{\varphi}}=0 $.\\
A second type of line is $\cl=\langle (0,0,0,1),(0,1,u,0)\rangle$, $u\ne\infty$.
This line in the osculating plane $[1:0:0:0]$,  is mapped to $\langle (0,0,0,1),(0,1,(u+2c)/a,1)$
by using $M_{a,c}$.\\
Consider first the case that $q$ is odd. Choosing $c=-u/2$ gives:\\
\indent
$\co_4$: true unisecants in an osculating plane with $\cl=\langle (0,0,0,1),(0,1,0,0)\rangle$
and this also shows that they form a single orbit of size $q(q + 1)$.
So $H$ has rows $(1,0,0,0),(0,0,1,0)$.
Hence  $\varphi(x)=x^3/x$, $\tilde{\varphi}(x)=x^2$, and $\varphi $ has base divisor $P(0)$, and
$R_{\varphi }= D_{\varphi }=P(0)+P_{\infty}$.\\
We continue with the case that $q$ is even.
Now $\langle (0,0,0,1), (0,1,u/a,0)\rangle$ is the image of $\cl$ under the map $M_{a,c}$.
If $u=0$ we find the same as in the case $q$ odd. If $u\ne0$ we get $\cl=\langle (0,0,0,1),(0,1,1,0)\rangle$.
This gives two orbits:\\
\indent
$\co_4^-(2)$: with representative $\cl=\langle (0,0,0,1),(0,1,0,0)\rangle$. This orbit has size $q+1$.
So $H$ has rows $(1,0,0,0),(0,0,1,0)$.
Hence  $\varphi(x)=x^3/x$ has base divisor $P(0)$, and $\tilde{\varphi}(x)=x^2$ is purely inseparable.\\
\indent
$\co_4^+(2)$: with representative $\cl=\langle (0,0,0,1),(0,1,1,0)\rangle$; of size $q^2-1$.
So $H$ has rows $(1,0,0,0),(0,1,1,0)$.
Hence  $\varphi(x)=x^3/(x^2+x)$, $\tilde{\varphi}(x)=x^2/(x+1)$,
and $\varphi $ has base divisor $P(0)$, and $R_{\tilde{\varphi}}= P(0)$, $D_{\tilde{\varphi}}= 2P(0)$.  \\
The third type of line is $\langle (0,0,0,1),(1,u,v,0)\rangle$, corresponding essentially to:
\indent
$\co_5$: unisecants not in an osculating plane.\\
This line is mapped by $M_{a,c}$ to $\langle (0,0,0,1),(1,(u+c)/a,(c^2+2cu+v)/a^2,0)\rangle$
(by making the last coordinate 0) and we now take $c=-u$ and obtain
$\langle (0,0,0,1),(1,0,(v-u^2)/a^2,0)\rangle$. This gives the following two cases:\\
If $v=u^2$, then $\cl $ is the secant through $P(0)$ and $P(\infty)$, so we have already
seen these lines.\\
If $v\ne u^2$, let $w=(u^2-v)/a^2\ne0$ and $d =w^{-1}$. Choosing different $a$'s does not change
the quadratic character of $d$, hence we get $\cl = \langle (0,0,0,1),(-d,0,1,0)\rangle$ with
$d\ne0$ being a square or a non-square if $q$ is odd, and one case if $q$ is even. Consider the two cases if $q$ is odd:\\
\indent
$\co_5^-$: $d$ is a non-square, with representative $\cl = \langle (0,0,0,1),(-d,0,1,0)\rangle$.
This orbit has size $\frac12q(q^2-1)$. So $H$ has rows $(1,0,d,0),(0,1,0,0)$.
Hence  $\varphi(x)=(x^3+dx)/x^2$, $\tilde{\varphi}(x)=(x^2+d)/x$,
and $\varphi $ has base divisor $P(0)$, and $R_{\tilde{\varphi}}=D_{\tilde{\varphi}}=Q_d$.\\
\indent
$\co_5^+$: $d$ is a non-zero square, we can take $d=1$ with representative $\cl = \langle (0,0,0,1),(-1,0,1,0)\rangle$.
This orbit has size $\frac12q(q^2-1)$, too.
So  $\varphi(x)=(x^3+x)/x^2$, $\tilde{\varphi}(x)=(x^2+1)/x$,
and $\varphi $ has base divisor $P(0)$, and $R_{\tilde{\varphi}}=D_{\tilde{\varphi}}=P(1)+P(-1)$.\\
\indent
$\co_5(2)$: if $q$ is even every non-zero element is a square and we take $d=1$ with representative $\cl = \langle (0,0,0,1),(1,0,1,0)\rangle$.
This orbit has size $q(q^2-1)$.
So  $\varphi(x)=(x^3+x)/x^2$, $\tilde{\varphi}(x)=(x^2+1)/x$,
and $\varphi $ has base divisor $P(0)$, and $R_{\tilde{\varphi}}=P(1)$ and $D_{\tilde{\varphi}}=2P(1)$.\\
\indent
$\co_5'$: passants in an osculating plane, $p\not=3$. \\
We take our favourite osculating plane $\pi=[1:0:0:0]$ at the point
$P=(0:0:0:1)$. The stabilizer group $G_{q,P}$ of $P$ under $G_q$ is as before and has size $q(q-1)$. In this plane
we take our favourite external line: $\cl_v=\langle (0,0,1,0), (0,v,0,1)\rangle \subseteq \pi$.
It is easy to check that the stabilizer of $\cl_v$ under $G_{q,P}$ is generated by ${diag}(1,-1,1,-1)$, and $\cl_v$ and $\cl_{a^2v}$ are in the same orbit.
So the orbit of $\cl_v$ under $G_{q,P}$ has size $q(q-1)$ if $q$ is even and $\frac12q(q-1)$ if $q$ is odd.
Now $H$ has rows $(1,0,0,0),(0,1,0,-v)$.
Hence  $\varphi(x)=\tilde{\varphi}(x)=x^3/(x^2-v)$, and $\varphi'(x)=x^2(x^2-3v)/(x^2-v)^2$.
If $q$ is odd, then $\cl_{u}$ and $\cl_v$ such that $3u$ is a non-zero square and $3v$ is a non-square  are in two different orbits
and together they are all external lines in $\pi$.\\
\indent
$\co_5'^-$: with $3v$ a non-square with representative $\cl =\langle (0,0,1,0), (0,v,0,1),\rangle$. This orbit has size $\frac12q(q^2-1)$, and
$R_{\varphi}=D_{\varphi}=2P(0)+Q_{3v}$.\\
\indent
$\co_5'^+$: with $3v$ a non-zero square, we take $v=\frac13$ with representative $\cl =\langle(0,\frac13,0,1),(0,0,1,0)\rangle$.
This orbit has size $\frac12q(q^2-1)$, and $\varphi(x)=\tilde{\varphi}(x)=x^3/(x^2-\frac13)$ and
$R_{\varphi}=D_{\varphi}=2P(0)+P(1)+P(-1)$.\\
\indent
$\co_5'(2)$: if $q$ is even every non-zero element is a square and we take $v=1$  with representative $\cl =\langle(0,1,0,1),(0,0,1,0)\rangle$.
This orbit has size $q(q^2-1)$, and $\varphi(x)=\tilde{\varphi}(x)=x^3/(x+1)^2$, and $R_{\tilde{\varphi}}=2P(0)+P(1)$ and $D_{\tilde{\varphi}}=2P(0)+2P(1)$.\\
\indent
$\co_6=\co_6'$: true passants not in an osculating plane.\\
Let $\varphi $ be a rational function in this class. Then $\varphi $ is a morphism, since the corresponding line is a passant so it does not intersect $\C_3$.
The ramification exponents $e_P(\varphi) $ for all places $P$ are at most $2$, since the passant is not in an osculating plane by Remark \ref{r-ramexp-intmult}.
Hence $R_{\varphi}$ is simple.
Moreover $\varphi $ does not ramify at two distinct points in a fibre $\varphi^{-1}(Q)$ for all places $Q$ by Proposition \ref{p-deg-ram-ind}.
Hence $\varphi $ is a simple morphism. \\
The morphism $\varphi :\proj^1 \rightarrow \proj^1$ of degree $3$ gives an extension $L$ of degree $3$ of $K=\fq(x)$, the field of rational functions in one variable, and there exists a unique intermediate field $S$, $K \subseteq S \subseteq L$, such that $S/K$ is separable and $L/S$ is purely
inseparable, see \cite[Appendix 8]{stichtenoth:1993}. If the extension degrees of $K \subseteq S \subseteq L$ are $s$ and $l$, respectively
then $sl=3$ the degree of the extension $L/K$.
So either $S=L$ and $\varphi $ is separable, or $K=S$ and $L/K$ is purely inseparable and $p=3$, so $\varphi(x)=x^3$ after a $RL$-transformation which is case $\co_7(3)$.
Therefor $\varphi $ is a separable simple morphism.\\
\indent
$\co_7(3)$: the axis of $\bg_3$, $p=3$ with $\cl=\langle(0,1,0,0),(0,0,1,0)\rangle $.
So $H$ has rows $(1,0,0,0),(0,0,0,1)$.
Hence  $\varphi(x)=\tilde{\varphi}(x)=x^3$ and $\varphi $ is purely inseparable.\\
\indent
$\co_8(3)$: passants meeting the axis, $p=3$.\\
Every plane containing the axis is an osculating plane.
So every line meeting the axis is in an osculating plane. We may take as osculating plane $\pi(\infty)=[0:0:0:1]$, that is given by $X_3=0$.
Let $\cl$ be a passant contained in $\pi(\infty)$ meeting the axis given by $X_0=X_3=0$ at the point $P=(0:u:v:0)$.
All $\varphi $ in $G_q$ leave the axis invariant. If $\varphi \in G_q$  leaves  $\pi(\infty)$ invariant, then it fixes also $P(\infty)$.
So $\varphi (x)=ax+b$, that is with $c=0$ and $d=1$, and  $\varphi $ leaves $c(\infty,\infty)$,
that is the tangent line of $\pi(\infty)$ given by $X_2=X_3=0$ invariant.
Hence $\varphi $ leaves the intersection of the axis and $c(\infty,\infty)$ invariant. So it leaves $P_1=(0:1:0:0)$ invariant.
Indeed $\varphi (0:u:v:0)=(0:au-bv:v:0)$. So $\varphi (P_1)=P_1$, and for all $v\ne0$ there exists a $\varphi $ with $a=1$ and $b=u/v$
such that  $\varphi (0:u:v:0)=P_2=(0:0:1:0)$. Therefor we may assume that the passant meets the axis in $P_1$ or $P_2$.
This gives two cases:\\
Passants in $\pi(\infty)$ through $P_1$ are given by $\cl_{1,u}=\langle ((0,1,0,0), (u,0,1,0),\rangle$ with $u\ne 0$.
The transformation $\varphi(x)=ax$ with $a=1/u$ maps $\cl_{1,u}$ to $\cl_{1,a^2u}$ which gives two orbits, since $q$ is odd:\\
\indent
$\co_{8.1}(3)^-$: $u$ a non-square with representative $\cl=\langle (0,1,0,0), (u,0,1,0),\rangle$.
This orbit has size $\frac12(q+1)(q-1)$.
So $H$ has rows $(1,0,-u,0),(0,0,0,1)$.
Hence  $\varphi(x)=\tilde{\varphi}(x)=x^3-ux$ and $R_{\varphi}=2P(\infty)$, $D_{\varphi}=4P(\infty)$.\\
\indent
$\co_{8.1}(3)^+$: $u$ a non-zero square, we can take $u=1$ with representative $\cl=\langle (0,1,0,0), (1,0,1,0),\rangle$.
This orbit has size $\frac12(q+1)(q-1)$.\\
Passants in $\pi(\infty)$ through $P_2$ are given by $\cl_{2,v}\langle (0,0,1,0), (v,1,0,0)\rangle$ with $v\ne 0$.
The transformation  $\varphi(x)=ax$  maps $\cl_{2,v}$ to $\cl_{2,av}$  which gives one orbit:\\
\indent
$\co_{8.2}(3)$: with representative  $\cl=\langle (0,0,1,0),(1,1,0,0)\rangle$. This orbit has size $(q+1)(q^2-q)$.
So $H$ has rows $(1,-1,-0,0),(0,0,0,1)$.
Hence  $\varphi(x)=\tilde{\varphi}(x)=x^3-x^2$ and $R_{\varphi}=P(0)+2P(\infty)$, $D_{\varphi}=P(0)+3P(\infty)$.\\
\end{proof}

\begin{rem}
$|\co_6|=q(q-1)(q^2-1)$ and $|G_q|=q(q-1)(q+1)$. Hence $\co_6$ is subdivided in at least $q-1$ orbits.
Without proof we state that if $p\ne 2$ and $p\ne 3$, then $\co_6$ is subdivided in $5$ subclasses with ramification divisors
$P_1+P_2+P_3+P_4$, $P_1+P_2+Q$, $Q_1+Q_2$, $P+R$ and $S$, where the $P_i$ and $P$ are places of degree $1$, the $Q_i$ and $Q$ are places of degree $2$,
and $R$ and $S$ are places of degree $3$ and $4$, respectively.
Moreover the cross-ratio of the $4$ points over $\bar{\fq}$ of the ramification divisor determines the orbit.
\end{rem}

\begin{rem}
All cases of Theorem \ref{t-lines} except $\co_6$ are obtained in \cite[Theorem 3.1]{davydov:2021b} and cited also in \cite[Theorem 2.3]{davydov:2021a}.
Our classification is in agreement with the results of \cite{davydov:2021b}.
\end{rem}

\subsection{The determination of $\mu_q$}\label{s-det-muq}
In the table of the following proposition ($q\geq 23$) rows indicate the classes of lines and column headers
indicate $q$ mod $6$.

\begin{prop}\label{p-table}Let $q\geq 23$.
Then the entries in the following table indicate whether a case contributes to $\mu_q$ by a plus sign, and by a minus sign otherwise.
\[
\begin{array}{|c|r|c|c|c|c|c|}\hline
\text{Class}   & \text{Size}                & 1(6) & 2(6) & 3(6) & 4(6) & 5(6) \\ \hline
\co_1          &\frac12q^2+\frac12q    & -    & -    & -    & -    & -    \\
\co_1'         &\frac12q^2+\frac12q    & +    & -    &      & +    & -    \\
\co_2          & q+1                   & -    & -    & -    & -    & -    \\
\co_3          &\frac12q^2-\frac12q    & -    & -    & -    & -    & -    \\
\co_3'         &\frac12q^2-\frac12q    & -    & +    &      & -    & +    \\
\co_4          & q^2+q                 & +    &      & +    &      & +    \\
\co_4^-(2)     & q+1                   &      & -    &      & -    &     \\
\co_4^+(2)     & q^2-1                 &      & +    &      & +    &      \\
\co_5          & q^3-q                 & +    & +    & +    & +    &+     \\
\co_5'         & q^3-q                 & +    & +    &      & +    & +    \\
\co_6          & q^4-q^3-q^2+q         & +    & +    & +    & +    & +    \\
\co_7(3)       & 1                     &      &      & -    &      &      \\
\co_{8.1}^-(3) &\frac12q^2-\frac12     &      &      & -    &      &      \\
\co_{8.1}^+(3) &\frac12q^2-\frac12     &      &      & +    &      &      \\
\co_{8.2}(3)   & q^3-q                 &      &      & +    &      &      \\ \hline
\end{array}
\]
\end{prop}

\begin{proof}The partition of Proposition \ref{t-lines} is used. Here the different cases are considered by increasing degree of $\tilde{\varphi }$.\\
$(1)$ If $\deg(\tilde{\varphi })=1$, then the base divisor has degree $2$ and $\cl $ is a chord or a tangent: $\co_1$, $\co_2$ or $\co_3$.\\
\indent
$\co_1$: Real chords are in $3$-planes, but do not contribute to $\mu_q$,
since the points on these lines contribute already to $a_1(T)$ or $a_2(T)$.\\
\indent
$\co_2=\co_2'$: A plane that contains a tangent line at $P$ of $\C_3(q)$,
intersects $\C_3$ in the divisor $2P+P'$ where $P'$ is another point of $\C_3(q)$.
Hence tangent lines are not contained in a $3$-plane.\\
\indent
$\co_3$: A plane that contains an imaginary chord at $Q$,
intersects $\C_3$ in the divisor $Q+P$ where $P$ is a point of $\C_3(q)$ and $Q$ a place of degree $2$.
Hence imaginary chords are not contained in a $3$-plane.\\
$(2)$ If $\deg(\tilde{\varphi })=2$, then the base divisor is a place $P_1$ of degree $1$ and $\cl $ is a unisecant: $\co_4$ or $\co_5$.\\
\indent
$\co_4^-(2)$: In this case $\tilde{\varphi }$ is purely separable and $\tilde{\varphi }^{-1}(x)$ consists of one point, for all $x$.
Hence there are {\em no} $3$-planes containing $\cl $.\\
In all other subcases of $\co_4$ or $\co_5$ the morphism $\tilde{\varphi }$ is separable by Propositions \ref{t-lines} and \ref{p-deg=2}.
Hence there is an $\fq$-rational points $x$ on $\proj^1$ such that $\tilde{\varphi }^{-1}(x)$
consists of two $\fq$-rational points $P_2(x)$ and $P_3(x)$ which are distinct from $P_1$ by Proposition \ref{p-deg=2pijk}.
So apart from $P_1$, that is in all planes containing $\cl$, there is a $3$-plane that contains $P_1$, $P_2(x)$ and $P_3(x)$.\\
$(3)$ If $\deg(\tilde{\varphi })=3$, then $\varphi= \tilde{\varphi })$ has no base points and $\cl $ is an axis or a passant.
These are the remaining cases of Proposition \ref{t-lines}.\\
\indent
$\co_1'$: real axes with representative rational function  $\varphi(x)=x^3$ and corresponding line $\cl $. So $\Delta_{\varphi}(x,y) = x^2+xy+y^2$\\
If $q=1$ mod $3$, then the double point scheme $\E_{\varphi }$ contains $(x,\omega x)$ and $(x,\omega^2 x)$ with $\omega^3=1$ and $\omega\ne 1$.
Hence there is a $3$-plane containing $\cl $ and the three points $P(x)$, $P(\omega x)$ and $P(\bar{\omega} x)$ if $x\ne0$ and $x\ne \infty $.
So we get a contribution to $\mu_q$.
Furthermore $\E_{\varphi }$ is reducible over $\fq$ containing two components of bidegree $(1,1)$ that intersect in $(0,0)$ and $(\infty, \infty)$.
If $q=-1$ mod 3, then $\E_{\varphi }$ has no $\fq$-rational points except $(0,0)$ and $(\infty, \infty)$ and there is no contribution to $\mu_q$.
$\E_{\cl }$ is irreducible over $\fq$, but reducible over $\fqtwo$ with two components that are conjugate and intersect in $(0,0)$ and $(\infty, \infty)$.\\
\indent
$\co_3'$: Imaginary axes, $p\not=3$, with representative  rational function $\varphi(x)=(x^3-3nx-n)/(x^2+x+\frac13(1-n))$ and corresponding line $\cl $.
where $x^2+x+n$ is irreducible, that is the discriminant $1-4n$ is a non-square if $q$ is odd, and $tr(n)=1$ if $q$ is even.
Then
$$
\textstyle \Delta_{\varphi} (x,y)=x^2y^2+xy(x+y)+\frac13(1-n)(x^2+xy+y^2)+3xy+n(x+y)+n^2.
$$
Let $\xi$ and $\bar{\xi}$ be the roots of $x^2+x+n$.
Consider the line $\cl$ and the point $P(x)$ on $\C_3(q)$.
Under the null-polarity $\cl$ and $P(x)$ are mapped to $\cl'$ and $P'(x)$, respectively,
where $\cl'$ is an imaginary chord of $\bg_3$. So $\cl'$ intersects $\bg_3$ in the conjugate points $P'(\xi)$ and $P'(\bar{\xi})$.
There exists a fractional transformation $\varphi \in G_{q^2}$, that is with coefficients in $\fqtwo $
such that $\varphi (\xi)=\xi $, $\varphi (\bar{\xi})=\bar{\xi} $ and $\varphi (x)=0$.
Then $\bar{\varphi}(x)=0$, $\bar{\varphi}(\bar{\xi})$ is the conjugate of $\varphi(\xi)$ which is $\bar{\xi}$, and similarly
$\bar{\varphi}(\xi)=\xi$. So $\bar{\varphi}=\varphi $, since $G_{q^2}$ acts sharply $3$-transitive on $\proj^1(q^2)$.
Hence $\varphi \in G_q$ and we assume without loss of generality that $x=0$.\\
Now $\Delta_{\varphi} (0,y)= \frac13(1-n)y^2+ny+n^2$.
This quadratic polynomial has discriminant $-3(1-4n)n^2/9$.
If $q$ is odd there are two distinct solutions if $-3$ is a non-square, since $1-4n$ is a non-square.
So there is $3$-plane containing $\cl$ if  $q \equiv 2 \mod 3$, and there is no such $3$-plane if $q \equiv 1 \mod 3$.
If $q$ is even, the quadratic equation becomes $y^2+y+n+1=0$,  and we find a $3$-plane for some $y$ if
the trace of $ac/b^2$ is $0$, where $a=1$, $b=1$ and $c=n+1$. So $tr(n+1)=0$.
Hence  $tr(1)=1$, since $tr(n)=1$. This again is the case if and only if $q \equiv 2 \mod 3$. \\
\indent
$\co_5'$: Passants in an osculating plane, $p\not=3$, with representative rational function  $=x^3/(x^2-v)$ and corresponding line $\cl $. Then
$$
\Delta_{\varphi }(x,y) = x^2y^2-v(x^2+xy+y^2).
$$
We first consider the case that $q$ is odd.  The discriminant of $\Delta_{\varphi }(x,y)$ as polynomial in $y$ is $vx^2(4x^2-3v)$.
This discriminant is a square if and only if $4vx^2-3v^2-u^2=0$ has a $\fq$-rational solution $(x,u)$.
The projective curve with equation $4vx^2-3v^2z^2-u^2=0$ in the variables $x$, $u$ and $z$ with parameter $v$ defines a nonsingular conic
with $q+1$ $\fq$-rational points, with at most $2$ points where $z=0$, at most $2$ points for which $u=0$,
at most $2$ points leading to a solution $x=y$. So for $q>6$ there is an $x\in \fq$ such that the discriminant is a non-zero square
giving two solutions of $\Delta_{\varphi }(x,y)=0$ in $y$ which are distinct from $x$.
So there is a $3$-plane that contain the line $\cl$.\\
\indent
$\co_5'(2)$: is the subclass of $\co_5'$ with $q$ even. In this case $v$ is a square and we can take $v=1$. We want
$\tr(ac/b^2)=0$ with $a=x^2+1$, $b=x$ and $c=x^2$, so $\tr(x^2+1)=0$.
Now the map $x\mapsto x^2+1$ is a bijection, so it has trace $0$ for $\frac12q$ values of $x$.
So we get a contribution to $\mu_q$ and the number of $3$-planes containing $\cl_1$ is  $1+\frac12q$.\\
Hence in all subcases of $\co_5'$ we get a contribution to $\mu_q$.\\
\indent
$\co_6=\co_6'$: true passants not in an osculating plane.
Let $\cl $ be a line in this class. The corresponding rational function $\varphi $ is separable and simple by Proposition \ref{t-lines}.
Hence $\E_{\varphi  }$ is a curve of genus $1$ by Corollary \ref{c-elliptic}.
Furthermore for $q\geq 23$ there exist three mutually distinct elements $x,y,z$ in $\proj^1(q) $
such that $\varphi(x) =\varphi(y)=\varphi (z)$ by Remark \ref{r-elliptic}.
Hence $P(x)$, $P(y)$ and $P(z)$ determine a $3$-plane containing $\cl $.\\
\indent
$\co_7(3)$: The axis of $\bg_3$, $p=3$. The pencil of planes containing the axis consists of all osculating planes.
Hence the axis does not lie on a $3$-plane.\\
\indent
$\co_8(3)$: Passants meeting the axis, $p=3$. This class has three orbits:
\indent
$\co_{8.1}(3)$: The representative rational function is $\varphi(x)=x^3-ux$ with corresponding line $\cl $.
Then  $\Delta_{\varphi }(x,y)=x^2+xy+y^2-u$ and $\Delta_{\varphi }(x,y)=3x^2-u=-u\ne 0$, since $p=3$.
The discriminant of $\Delta_{\varphi }(x,y)$ as polynomial in $y$ is $x^2-4(x^2-u)=u$.\\
\indent
$\co_{8.1}(3)^-$: This is the subcase with $u$ a non-square. Hence $\Delta_{\varphi }(x,y)=0$ has no solutions in $y$ for all $x$.
The point  $x=\infty $ corresponds with a plane tangent to $\C_3$ at $P(\infty)$, which is not a $3$-plane.
Hence there are no $3$-planes containing $\cl $.\\
\indent
$\co_{8.1}(3)^+$: This is the subcase with $u=1$ a non-zero square.
Then the discriminant is $1$.
Hence there are two solutions $y=x\pm1$ which are distinct from $x$.
Therefore there are $3$-planes containing $\cl $.\\
\indent
$\co_{8.2}(3)$: The representative rational function is $\varphi(x)=x^3-x^2$ with corresponding line $\cl $.
Then  $\Delta_{\varphi }(x,y)=x^2+xy+y^2-x-y$ and $\Delta_{\varphi }(x,x)=3x^2-2x=x$, since $p=3$. So, if $\Delta_{\varphi }(x,x)=0$, then $x=0$.
The discriminant of $\Delta_{\cl}(x,y)$ as polynomial in $y$ is  $(x-1)^2-4(x^2-x)=1-x$ must be a non-zero-square.
If $q>3$, then there is a $x\in \fq\setminus \{0,1 \}$ such that $1-x $ is a non-zero square, and $\Delta_{\varphi }(x,y)=0$ has two solutions in $y$ not equal to $x$.
Hence there is a $3$-plane containing $\cl $.
\end{proof}

\begin{rem}
Theorem \ref{p-table} was enough to solve our problem, that is to know wether a line of a given class is contained in a $3$-plane or not.
In \cite[Theorem 3.3]{davydov:2021a} a more detailed result is given: \\
a) For all classes of lines, including $\co_6$, the exact number is computed  of lines of a given class that are contained in a plane of a given class.\\
b) For all classes of lines, apart from $\co_6$, the exact number is computed of planes of a given class through a line of a given class.\\
c) For the lines of $\co_6$,  the average number is computed of planes of a given class through a line of $\co_6$.\\
So in fact for all cases, apart from $\co_6$ we could have referred to \cite[Theorem 3.3]{davydov:2021a} instead.
\end{rem}

\begin{rem}\label{r-permutation-fct-deg3}
The permutation rational functions of degree $3$ are classified in \cite{ferraguti:2020}.
There are $6$ of them and they confirm the findings in the table of Proposition \ref{p-table}:
$\co_3'$ for $q \equiv 1 \mod 6$, $\co_1'$ for $q \equiv 2 \mod 6$,
$\co_7(3)$ and $\co_{8.1}(3)^-$ for $q \equiv 3 \mod 6$,
$\co_3'$ for $q \equiv 4 \mod 6$, and $\co_1'$ for $q \equiv 5 \mod 6$.
\end{rem}

We summarize our findings in the following.

\begin{thm}\label{t-muq}
If $q\geq 23$, then
\[
\mu_q=\left\{\begin{array}{lll}
  q^4+q^3+\frac12q^2+\frac12q     &\text{if} & q=1\mod6 \\[4pt]
  q^4+q^3+\frac12q^2-\frac32q-1   &\text{if} & q=2\mod6 \\[4pt]
  q^4+q^3+\frac12q^2-\frac12                     &\text{if} & q=3\mod6 \\[4pt]
  q^4+q^3+\frac12q^2-\frac12q-1   &\text{if} & q=4\mod6 \\[4pt]
  q^4+q^3+\frac12q^2-\frac12q     &\text{if} & q=5\mod6 \\[4pt]
  \end{array}\right.
\]
\end{thm}

\begin{proof}
This follows from Proposition \ref{p-table} by adding up the sizes of the corresponding entries in the second column
if  there is a plus sign in the corresponding row and column of $i\mod 6$.
\end{proof}

\section{Conclusion}

The extended coset leader weight enumerator of the generalized Reed-Solomon $[q+1,q-3,5]_q$ code is computed for $q\geq 23$.
For this we need to refine the known classification \cite{bruen:1977,hirschfeld:1985} of the points, lines and planes in the projective three space under the action of projectivities that leave the twisted cubic invariant. The given classification is complete except for the class  $\co_6$ of true passants not in an osculating plane.
The refined classification and the line-plane incidence, apart from $\co_6$ are also obtained in \cite{davydov:2021a,davydov:2021b}. \\
The relation between codimension $2$ subspaces of $\proj^r$ and rational functions of degree at most $r$ is given.\\
Furthermore the double point scheme $\E_{\varphi} $ of a rational function $\varphi$ is studied in general.
If the rational function $\varphi$ is a separable simple morphism of degree $d$, then $\E_{\varphi} $ is an absolutely irreducible  curve of genus $(d-1)^2$.
In particular, the pencil of a true passant of the twisted cubic, not in an osculating plane  gives a curve of genus $1$ as double point scheme.\\
In order to compute the (extended) list weight enumerator \cite{jurrius:2015} of this code is beyond the scope of this article,
since one needs to know the distribution of the numbers of $\fq $-rational points of the double point schemes of all the passants not in an osculating plane.
The complete classification of all orbits of lines will be given in an upcoming article.

\section{Acknowledgment}
We found out that the main result of the classification of Theorem \ref{t-lines} on lines in three space over finite fields
and Proposition \ref{p-table} on the line-plane incidence, with the exception $\co _6$  was also obtained in recent papers by \cite{davydov:2021a,davydov:2021b}
that were submitted one week earlier on arXiv than this paper. We thank the authors for showing us their work.\\
The third author was partly supported by Grant K 124950 of the Hungarian National Research, Development
and Innovation Fund. In his case, also the "Application Domain Specific Highly Reliable IT Solutions” project  has  been implemented with the support provided from the National Research,  Development and Innovation Fund of Hungary, financed under the Thematic  Excellence Programme TKP2020-NKA-06 (National Challenges Subprogramme) funding scheme.

\bibliographystyle{plain}

\end{document}